\documentclass{article}

\usepackage[utf8]{inputenc}
\usepackage[italian,english]{babel}

\usepackage{geometry}								% Per modificare l'impaginazione
\usepackage{indentfirst}  		% Aggiunge l'indentazione all'inizio delle sezioni
\usepackage[linktoc=all]{hyperref}	% Per i collegamenti ipertestuali. [linktoc=all] mette i link nell'indice
\usepackage{tikz}				% Per fare tikz pictures

\usepackage{amsmath}
\usepackage{amssymb}	% Carica anche {amsfonts}
\usepackage{mathrsfs}	% Per \mathscr
\usepackage{amsthm}	% Per \newtheorem

\theoremstyle{plain}
\newtheorem{mythm}{Theorem}[section]
\newtheorem{myprop}[mythm]{Proposition}
\newtheorem{mylemma}[mythm]{Lemma}
\newtheorem{mycor}[mythm]{Corollary}
\newtheorem{mydef}[mythm]{Definition}
\newtheorem{myquest}[mythm]{Question}

\newtheorem{introthm}{Theorem}

\theoremstyle{remark}

\newtheorem*{myrmk}{Remark}
\newtheorem*{myobs}{Observation}

\newcommand{\ol}[1]{\overline{#1}}
\newcommand{\ot}[1]{\widetilde{#1}}

\newcommand{\rar}{\rightarrow}

\newcommand{\gen}[1]{\langle#1\rangle}
\newcommand{\ggen}[1]{\langle\langle#1\rangle\rangle}
\newcommand{\abs}[1]{|#1|}
\newcommand{\restr}[2]{#1\big|_{#2}}
\newcommand{\bR}{\mathbb R}

\newcommand{\bZ}{\mathbb Z}
\newcommand{\bN}{\mathbb N}

\newcommand{\fI}{\mathfrak I}
\newcommand{\sgr}{\le}
\newcommand{\nor}{\trianglelefteq}
\newcommand{\rank}[1]{\normalfont{\text{rank}}(#1)}
\newcommand{\im}[1]{\normalfont{\text{im}}(#1)}
\newcommand{\core}[1]{\normalfont{\text{core}}(#1)}
\newcommand{\bcore}[1]{\normalfont{\text{core}}_*(#1)}
\newcommand{\cov}[1]{\normalfont{\text{cov}}(#1)}
\newcommand{\fold}[1]{\normalfont{\text{fold}}(#1)}
\newcommand{\grafo}{F_n\text{-labeled graph}}
\newcommand{\grafos}{F_n\text{-labeled graphs}}

\newcommand{\parallele}{\text{parallel}}
\newcommand{\inserzione}{\text{insertion word}}
\newcommand{\omeno}[1]{o^-(#1)}
\newcommand{\opiu}[1]{o^+(#1)}

\title{Ideals of equations for elements in a free group and Stallings folding}
\author{
Dario Ascari \thanks{Ascari was funded by the Engineering and Physical Sciences Research Council.}\\
{\small \textit{Mathematical Institute, Andrew Wiles Building,}}\\
{\small \textit{University of Oxford, Oxford OX2 6GG, UK}}\\
{\small e-mail: \texttt{ascari@maths.ox.ac.uk}}\\
\\
}
\begin{document}

\maketitle

\begin{abstract}
Let $F$ be a finitely generated free group and let $H\sgr F$ be a finitely generated subgroup. Given an element $g\in F$, we study the ideal $\fI_g$ of equations for $g$ with coefficients in $H$, i.e. the elements $w(x)\in H*\gen{x}$ such that $w(g)=1$ in $F$. The ideal $\fI_g$ is a normal subgroup of $H*\gen{x}$, and we provide an algorithm, based on Stallings folding operations, to compute a finite set of generators for $\fI_g$ as a normal subgroup.

We provide an algorithm to find an equation in $\fI_g$ with minimum degree, i.e. an equation $w(x)$ such that its cyclic reduction contains the minimum possible number of occurrences of $x$ and $x^{-1}$; this answers a question of A. Rosenmann and E. Ventura. % (see \cite{Rosenmann1}).
More generally, we provide an algorithm that, given $d\in\bN$, determines whether $\fI_g$ contains equations of degree $d$ or not, and we give a characterization of the set of all the equations of that specific degree. We define the set $D_g$ of all integers $d$ such that $\fI_g$ contains equations of degree $d$; we show that $D_g$ coincides, up to a finite set, either with the set of non-negative even numbers or with the set of natural numbers.

Finally, we provide examples to illustrate the techniques introduces in this paper. We discuss the case where $\rank{H}=1$. We prove that both kinds of sets $D_g$ can actually occur. The examples also show that the equations of minimum possible degree aren't in general enough to generate the whole ideal $\fI_g$ as a normal subgroup.
\end{abstract}

\begin{center}
\end{center}

%\tableofcontents

\section{Introduction}

Given an extension of fields $K\subseteq F$ and an element $\alpha\in F$, a first interesting question to ask is to determine whether the element $\alpha$ is algebraic over $K$, i.e. whether it satisfies some non-trivial equation with coefficients in $K$. In other words we want to determine whether there exists a non-trivial polynomial $p(x)\in K[x]$ such that $p(\alpha)=0$. If the answer is affirmative, one tries to study the ideal $I_\alpha\subseteq K[x]$ of equations for $\alpha$ over $K$: this turns out to be a principal ideal, and thus its structure is very simple. Completely analogous questions can be asked in the context of group theory, but the answers turn out to be more complicated.

Let $F_n$ be a free group generated by $n$ elements $a_1,...,a_n$. Let $H\sgr F_n$ be a finitely generated subgroup and consider an infinite cyclic group $\gen{x}\cong\bZ$. An \textbf{equation} in $x$ with coefficients in $H$ is an element $w\in H*\gen{x}$ in the free product of $H$ and $\gen{x}$; $w$ has a unique expression as a reduced word in the alphabeth $\{x,x^{-1}\}\cup H\setminus\{1\}$.
%We indentify $w$ with the unique alternating word that represents it.
For an equation $w\in H*\gen{x}$ we define the \textbf{degree} of $w$ as the number of occurrences of $x$ and $\ol{x}$ in the cyclic reduction of $w$.

For an element $g\in F_n$, consider the map $\varphi_g:H*\gen{x}\rar F_n$ that is the inclusion on $H$, and that sends $x$ to $g$; this is the ``evaluation in $g$'' map. We say that $g$ is a \textbf{solution} for the equation $w$ if $\varphi_g(w)=1$. We define the \textbf{ideal} $\fI_g$ to be the normal subgroup $\fI_g=\ker\varphi_g$ of $H*\gen{x}$.

Fix a finitely generated subgroup $H\sgr F_n$ and an element $g\in F_n$. First of all, we would like to determine whether $\fI_g$ is trivial or not, i.e. whether $g$ satisfies some non-trivial equation over $H$ or not. This has been answered recently by A. Rosenmann and E. Ventura in \cite{Rosenmann1}, and in particular they obtained the following result:

\begin{mydef}
Let $H\sgr F_n$ be a finitely generated subgroup and let $g\in F_n$ be any element. We say that $g$ \textbf{depends} on $H$ if any of the following equivalent conditions hold:

(i) The ideal $\fI_g$ is non-trivial.

(ii) We have $\rank{\gen{H,g}}\le\rank{H}$.
\end{mydef}

\begin{mythm}
Let $H\sgr F_n$ be a finitely generated subgroup. Then there is an algorithm that computes a finite set of elements $g_1,...,g_k\in F_n$ such that, for every $g\in F_n$, the following are equivalent:

(i) The element $g$ depends on $H$.

(ii) The element $g$ belongs to one of the double cosets $Hg_1H,...,Hg_kH$.
\end{mythm}

For the whole paper, when an algorithm takes in input a finitely generated subgroup $H\sgr F_n$, we mean that the subgroup is given by means of a finite set of generators, each provided as a word in the basis $a_1,...,a_n$ of $F_n$.

\

In this paper, we study the structure of the ideal $\fI_g$ of the equations with coefficients in $H$ and with $g$ as a solution. Most of our results can be generalized to equations in more variables, but for simplicity of notation we deal only with the one-variable case; the statements of the results in more variables can be found in Section \ref{SectionMultivariate} at the end of the paper.

\

For an arbitrary homomorphism from a finitely generated free group to a finitely presented group, the kernel is always finitely generated as a normal subgroup. If the target is free, then it follows from Grushko's Theorem that there is an algorithm to find a finite normal generating set. In Section \ref{SectionStallings} we describe an efficient algorithm, with focus on the case of the map $\varphi_g:H*\gen{x}\rar F_n$ defined above, whose kernel is exactly the ideal $\ker\varphi_g=\fI_g$. The key idea, based on Stallings folding operations, is the following. In a chain of folding operations, the rank-preserving folding operations are homotopy equivalences, and thus isomorphisms at the level of fundamental group, while the non-rank-preserving folding operations give a non-injective map of fundamental groups (that means, in our case, adding generators to the kernel). The novel aspect of our algorithm, which is explained in section \ref{SectionStallings}, is the following: we show that the non-rank-preserving folding operations can be postponed until the end of the chain of folding operations (see Figure \ref{gmgl}): this gives a clean and efficient way to produce a set of generators for the kernel as a normal subgroup of $H*\gen{x}$.

\

In \cite{Rosenmann1} Rosenmann and Ventura ask the following question:

\begin{myquest}
Is it possible to (algorithmically) find an equation of minimum degree for an element $g$ that depends on $H$?
\end{myquest}

In Section \ref{SectionMain} we give an affirmative answer to this question.

\begin{introthm}[See Corollary \ref{algorithm}]\label{introthm1}
There is an algorithm that, given $H\sgr F_n$ finitely generated and $g\in F_n$ that depends on $H$, produces a non-trivial equation $w\in\fI_g$ of minimum possible degree.
\end{introthm}

We give a brief outline of the proof of Theorem \ref{introthm1}. We take a non-trivial (cyclically reduced) equation $w\in\fI_g$ of minimum possible degree; we think of $w$ as a word in the letters $a_1,...,a_n,x$ (where $a_1,...,a_n$ is a basis for $F_n$). We then prove that, for words of sufficient length, some parts of the word $w$ can be literally cut away, by means of a move that we call a ``parallel cancellation move'', introduced in Lemma \ref{parallelcancellation}; this produces another equation $w'\in\fI_g$, which is strictly shorter than $w$ and which has the same degree. By iterating this process, we prove that there is an equation in $\fI_g$ of minimum possible degree whose length is bounded (see Theorem \ref{main} for the precise bound). With this bound established, the algorithm now just takes all the (finitely many) elements of $H*\gen{x}$ which are short enough, and for each of them it checks whether it belongs to $\fI_g$, recording its degree.

A completely analogous result holds for equations of any fixed degree $d$: if $\fI_g$ contains a non-trivial equation of degree $d$, then it contains one whose length is bounded (see Theorem \ref{main3} for the precise bound). In particular we prove the following theorem:

\begin{introthm}[See Corollary \ref{algorithm3}]\label{introthm2}
There is an algorithm that, given $H\sgr F_n$ and $g\in F_n$ and an integer $d\ge1$, tells us whether $\fI_g$ contains non-trivial equations of degree $d$, and, if so, produces an equation $w\in\fI_g$ of degree $d$.
\end{introthm}

One of the interesting features of the algorithm is that the ``parallel cancellation moves'' of Lemma \ref{parallelcancellation} have inverses, namely the ``parallel insertion moves'' which we introduce in Lemma \ref{parallelinsertion}. This means that the arbitrary (cyclically reduced) equation of degree $d$, can be obtained from a short equation of degree $d$ by means of a finite number of insertion moves; this gives a characterization of all the equations of degree $d$ in terms of a finite number of short equations (see Theorem \ref{main4} for the details).

\

Next, we study the set $D_g=\{d\in\bN : \fI_g$ contains a non-trivial equation of degree $d\}$. We prove that $D_g$ coincides with either $\bN$ or $2\bN$ (the set of non-negative even numbers), up to a finite set. We provide an algorithm that, given $H$ and $g$, computes the set $D_g$.

\begin{introthm}[See Theorem \ref{degreeset}]\label{introdegreeset}
Exactly one of the following possibilities takes place:

(i) $D_g$ contains an odd number and $\bN\setminus D_g$ is finite.

(ii) $D_g$ contains only even numbers and $2\bN\setminus D_g$ is finite.
\end{introthm}

\begin{introthm}[See Theorem \ref{algorithm4}]
Given $H\sgr F_n$ finitely generated and $g\in F_n$ that depends on $H$, there is an algorithm that:

(a) Determines whether we fall into case (i) or (ii) of Theorem \ref{introdegreeset}.

(b) Computes the finite set $\bN\setminus D_g$ or $2\bN\setminus D_g$ respectively.
\end{introthm}

\

In Section \ref{SectionExamples}, we make use of the tools developed in the rest of the paper in order to work out explicit computations in some specific cases; in each case we compute the minimum degree $d_{min}$ for an equation in $\fI_g$ and the set $D_g$ of possible degrees. In example \ref{examplecyclic} we deal with the case where $\rank{H}=1$, showing that in this case $d_{min}$ is either $1$ or $2$. Examples \ref{example46} and \ref{example23} show that both cases of Theorem \ref{introdegreeset} can occur. One may be tempted to conjecture that the equations of $\fI_g$ of minimum possible degree $d_{min}$ are enough to generate the ideal $\fI_g$; we give counterexamples to this (see Examples \ref{example23} and \ref{example234}), showing that the ideal $\fI_g$ is not always generated by just the equations of degree $d_{min}$.

\

In a subsequent paper we will further investigate the properties of the ideal $\fI_g$.

\section*{Acknowledgements}

I would like to thank my supervisor Martin R. Bridson for useful comments and suggestions while working on the present paper.

\section{Preliminaries and notations}\label{Preliminaries}

%\subsection{Graphs}

With the word \textbf{graph} we mean a $1$-dimensional CW complex. We allow for multiple edges between the same pair of vertices, and we allow for edges from a vertex to itself. For a graph $G$ we denote with $V=V(G)$ the $0$-skeleton of $G$, and each point of $V$ is called \textit{vertex}; each connected component of $G\setminus V$ is called \textit{open edge} and its closure is called an \textit{edge}. A \textbf{combinatorial map} $f:G\rar G'$ between graphs is a continuous map which sends each vertex of $G$ to a vertex of $G'$, and each open edge of $G$ homeomorphically onto an open edge of $G'$.

For $l\ge1$ we define $I_l$  to be the graph obtained from a subdivision of the unit interval $[0,1]$ into $l$ arcs. More precisely, $I_l$ has $l+1$ vertices at $\frac{i}{l}$ for $i=0,...,l$, and $l$ edges given by closed intervals.

For $l\ge1$ we define $C_l$ to be the graph obtained from a subdivision of the unit circle $\{(x,y) : x^2+y^2=1\}\subseteq\bR^2$ into $l$ arcs. More precisely $C_l$ has $l$ vertices at the points $(\cos(\frac{2\pi i}{l}),\sin(\frac{2\pi i}{l}))$ for $i=0,...,l-1$, and $l$ edges given by closed arcs on the unit circle.

Let $G$ be a graph. A \textbf{combinatorial path} in $G$ is a combinatorial map $\sigma:I_l\rar G$ for some $l\ge1$; a \textbf{combinatorial loop} in $G$ is a combinatorial map $\sigma:C_l\rar G$ for some $l\ge1$. We say a combinatorial path (resp. loop) is \textbf{reduced} if it is locally injective. The local injectivity has to be checked only at the vertices of $I_l$ or $C_l$: in the interior of the edges, every combinatorial path/loop is locally injective by definition. We say that a combinatorial path $\sigma:I_l\rar G$ with $\sigma(0)=\sigma(1)$ is \textbf{cyclically reduced} if it is reduced when seen as a combinatorial loop.

\begin{figure}[h!]
\centering
\begin{tikzpicture}
%\node (0) at (0,0) {$C_6$};
%\node (1) at (2,0) {.};
%\node (2) at (1,1.73) {.};
%\node (3) at (-1,1.73) {.};
%\node (4) at (-2,0) {.};
%\node (5) at (-1,-1.73) {.};
%\node (6) at (1,-1.73) {.};
%\draw[out=90,in=330,looseness=0.7] (1) to (2);
%\draw[out=150,in=30,looseness=0.7] (2) to (3);
%\draw[out=210,in=90,looseness=0.7] (3) to (4);
%\draw[out=270,in=150,looseness=0.7] (4) to (5);
%\draw[out=330,in=210,looseness=0.7] (5) to (6);
%\draw[out=30,in=270,looseness=0.7] (6) to (1);
%
%\node (0a) at (5,0) {$C_8$};
%\node (1a) at (7,0) {.};
%\node (2a) at (6.4,1.4) {.};
%\node (3a) at (5,2) {.};
%\node (4a) at (3.6,1.4) {.};
%\node (5a) at (3,0) {.};
%\node (6a) at (3.6,-1.4) {.};
%\node (7a) at (5,-2) {.};
%\node (8a) at (6.4,-1.4) {.};
%\draw[out=90,in=315,looseness=0.7] (1a) to (2a);
%\draw[out=135,in=0,looseness=0.7] (2a) to (3a);
%\draw[out=180,in=45,looseness=0.7] (3a) to (4a);
%\draw[out=225,in=90,looseness=0.7] (4a) to (5a);
%\draw[out=270,in=135,looseness=0.7] (5a) to (6a);
%\draw[out=315,in=180,looseness=0.7] (6a) to (7a);
%\draw[out=0,in=225,looseness=0.7] (7a) to (8a);
%\draw[out=45,in=270,looseness=0.7] (8a) to (1a);

%\node (0b) at (-5.5,-0.5) {$I_5$};
\node (1b) at (-8,0) {.};
\node (2b) at (-7,0) {.};
\node (3b) at (-6,0) {.};
\node (4b) at (-5,0) {.};
\node (5b) at (-4,0) {.};
\node (6b) at (-3,0) {.};
\draw (1b) to (2b);
\draw (2b) to (3b);
\draw (3b) to (4b);
\draw (4b) to (5b);
\draw (5b) to (6b);
\end{tikzpicture}
\caption{The graph $I_5$.}
\label{clil}
\end{figure}
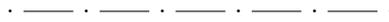

We are going to need the notion of core graph and of pointed core graph. Whenever we consider a graph $G$ with a basepoint, we always mean the basepoint to be a vertex.

\begin{mydef}
Let $G$ be a connected graph which is not a tree. Define its \textbf{core graph} $\core{G}$ as the subgraph given by the union of (the images of) all the reduced loops.
\end{mydef}

\begin{myobs}
Notice that $\core{G}$ is connected and every vertex has valence at least $2$.
\end{myobs}

%\begin{mydef}
%We say that a graph $G$ is \textbf{core} if $\core{G}=G$.
%\end{mydef}

\begin{mydef}
Let $(G,*)$ be a connected pointed graph which is not a tree. Define its \textbf{pointed core graph} $\bcore{G}$ as the subgraph given by the union of (the images of) all the reduced paths from the basepoint to itself.
\end{mydef}

\begin{myobs}
For a pointed graph $(G,*)$, there is a unique shortest path $\sigma$ (either trivial or embedded) connecting the basepoint to $\core{G}$; the graph $\bcore{G}$ consists exactly of the union $\core{G}\cup\im{\sigma}$.
\end{myobs}

We shall need to work explicitly with the following well-known construction. Let $T$ be a maximal tree contained in $G$, and let $E$ be the set of edges which are not contained in $T$; suppose we are given an orientation on each edge $e\in E$. For $e\in E$, there is a unique reduced path $\sigma_e$ in $G$ that starts at the basepoint, moves along $T$ to the initial vertex of $e$, crosses $e$ according to the orientation, and moves along $T$ from the final vertex of $e$ to the basepoint.

\begin{myprop}\label{fundamentalgroupgraph}
The fundamental group $\pi_1(G)$ is a free group with a basis given by the homotopy classes of the paths $\sigma_e$ for $e\in E$.
\end{myprop}
%\begin{proof}
%Let $q:G\rar G/T$ be the map that collapses the tree $T$ to a single point: since $T$ is contractible, the map $q$ induces an isomorphism of fundamental groups $q_*:\pi_1(G,*)\rar\pi_1(G/T,*)$. The space $G/T$ is a rose with petals corresponding to edges $e\in E$, and the path $q\circ\sigma_e$ goes around the petal corresponding to $e$. Thus the homotopy classes of paths $q_*([\sigma_e])=[q\circ\sigma_e]$ for a basis for the free group $\pi_1(G/T,*)$: it follows that the homotopy classes of paths $[\sigma_e]$ for a basis for the free group $\pi_1(G,*)$.
%\end{proof}

\subsection{Reduction of paths}

Consider the graph $I_l$ and for an edge $s$ of $I_l$ denote with $\omeno{s},\opiu{s}$ the vertices of $s$. Since $I_l$ is a subdivision of the unit interval, we adopt the convention that $\omeno{s}<\opiu{s}$ as points of the unit interval.

Let $G$ be a graph and $\sigma:I_l\rar G$ be a combinatorial path. If $\sigma$ is not reduced, then we can find two consecutive edges $s,t$ of $I_l$ such that $\sigma$ sends $s,t$ to the same edge $e$ of $G$, but crossed with opposite orientations. Let's say we have $\opiu{s}=\omeno{t}$: we consider the interval $s\cup t=[\omeno{s},\opiu{t}]$ and we collapse it to a point. We obtain a graph isomorphic to $I_{l-2}$, and we can define a map $\sigma':I_{l-2}\rar G$ which is equal to $\sigma$, except on the collapsed interval, where we set it to be equal to $\sigma(\omeno{s})=\sigma(\opiu{t})$. The map $\sigma':I_{l-2}\rar G$ is a combinatorial path, and it is homotopic to $\sigma$ (relative to the endpoints). If the path $\sigma'$ is not yet reduced, then we can reiterate the same process. This motivates the following definition:

\begin{mydef}\label{reductionprocess}
Let $G$ be a graph and let $\sigma:I_l\rar G$ be a combinatorial path. A \textbf{reduction process} for $\sigma$ is a sequence $(s_1,t_1),...,(s_m,t_m)$ with the following properties:

(i) $s_1,t_1,...,s_m,t_m$ are pairwise distinct edges of $I_l$.

(ii) For every $k=1,...,m$ we have $\omeno{s_k}<\omeno{t_k}$.

(iii) For every $k=1,...,m$, if we collapse each of $s_1,t_1,...,s_{k-1},t_{k-1}$ to a point, in the quotient graph the edges $s_k,t_k$ are adjacent.

(iv) For every $k=1,...,m$ the map $\sigma$ sends $s_k,t_k$ to the same edge of $G$ crossed with opposite orientations.
\end{mydef}

Think of $(s_k,t_k)$ as the $k$-th cancellation to be performed on the path $\sigma$. Condition (iii) says that, after performing the first $k-1$ cancellations, the edges $s_k,t_k$ are adjacent, ready to be canceled against each other. Condition (iv) ensures that $\sigma$ sends $s_k,t_k$ to the same edge of $G$ but with opposite orientations, so that the cancellation actually makes sense. Condition (ii) is just a useful convention, saying that the edges $s_k,t_k$ appear in this order on the unit interval $I_l$.

\begin{mylemma}\label{lemmino1}
Let $G$ be a graph and let $\sigma:I_l\rar G$ be a combinatorial path, together with a reduction process $(s_1,t_1),...,(s_m,t_m)$. Then for every $1\le \alpha<\beta\le m$ the edges $s_\alpha,t_\alpha,s_\beta,t_\beta$ appear on the interval in one of these orders: $s_\alpha,t_\alpha,s_\beta,t_\beta$ or $s_\beta,t_\beta,s_\alpha,t_\alpha$ or $s_\beta,s_\alpha,t_\alpha,t_\beta$.
\end{mylemma}
\begin{proof}
When we collapse $s_1,t_1,...,s_{\alpha-1},t_{\alpha-1}$ to a point, we have that $s_\alpha$ and $t_\alpha$ become adjacent. This means that $s_\beta,t_\beta$ can't both occur between $s_\alpha$ and $t_\alpha$. The conclusion follows.
\end{proof}
%
%\begin{mylemma}\label{lemmino2}
%Let $G$ be a graph and let $\sigma:I_l\rar G$ be a combinatorial path, together with a reduction process $(s_1,t_1),...,(s_m,t_m)$. Then for every $1\le\alpha\le m$ we have that the interval $[\opiu{s_\alpha},\omeno{t_\alpha}]$ is covered by the edges $\{s_1,t_1,...,s_{\alpha-1},t_{\alpha-1}\}$.
%\end{mylemma}

Let $\sigma:I_l\rar G$ be a path and let $(s_1,t_1),...,(s_m,t_m)$ be a reduction process for $\sigma$. If $2m<l$, then we can collapse each of the edges $s_1,t_1,...,s_m,t_m$ to a point in order to get a graph isomorphic to $I_{l-2m}$. We can define a continuous map $\sigma':I_{l-2m}\rar G$ which is equal to $\sigma$ on the edges which are not collapsed in the process. The map $\sigma'$ is a combinatorial path which is homotopic to $\sigma$ (relative to the endpoints), and it is called \textbf{residual path} of the cancellation process.

\begin{myprop}\label{extensionofprocess}
Let $\sigma:I_l\rar G$ be a combinatorial path and let $(s_1,t_1),...,(s_m,t_m)$ be a reduction process for $\sigma$. Then exactly one of the following holds:

(i) We have $2m=l$ and $\sigma(0)=\sigma(1)$ and $\sigma$ is nullhomotopic (relative to the endpoints).

(ii) We have $2m<l$ and the residual path $\sigma':I_{l-2m}\rar G$ is reduced.

(iii) There is a couple $(s_{m+1},t_{m+1})$ such that $(s_1,t_1),...,(s_m,t_m),(s_{m+1},t_{m+1})$ is a reduction process for $\sigma$.
\end{myprop}
\begin{proof}
Suppose that $2m<l$ and that the residual path $\sigma':I_{l-2m}\rar G$ is not reduced. Then there are two adjacent edges $s',t'$ in $I_{l-2m}$ such that $\sigma'$ sends $s',t'$ to the same edge of $G$, crossed with opposite orientation; let's also assume $\omeno{s'}<\omeno{t'}$. The domain $I_{l-2m}$ of $\sigma'$ is a quotient of the domain $I_l$ of $\sigma$; thus we find unique edges $s_{m+1},t_{m+1}$ of $I_l$ such that the quotient sends $s_{m+1},t_{m+1}$ to $s',t'$ respectively; notice that $\omeno{s_{m+1}}<\omeno{t_{m+1}}$ and that $s_{m+1},t_{m+1}$ are distinct, and they are also distinct from $s_1,t_1,...,s_m,t_m$. From the definition of $\sigma'$, it is immediate to see that $\sigma$ sends $s_{m+1},t_{m+1}$ to the same edge of $G$, crossed with opposite orientation. It follows that $(s_1,t_1),...,(s_m,t_m),(s_{m+1},t_{m+1})$ is a reduction process for $\sigma$, as desired.
\end{proof}

The above proposition essentially says that a reduction process can be inductively extended, until we get a path which is either trivial or reduced. A reduction process $(s_1,t_1),...,(s_m,t_m)$ is called \textbf{maximal} if it can't be extended by adding a couple of edges $(s_{m+1},t_{m+1})$, i.e. if it falls into case (i) or (ii) of Proposition \ref{extensionofprocess}. Of course every path admits at least one maximal reduction process. Despite the maximal reduction process not being unique in general, it turns out that the residual path is unique, as shown in the following proposition:

\begin{myprop}\label{reducedpath}
Let $[\sigma]$ be a non-trivial homotopy class of paths $\sigma:[0,1]\rar G$ (relative to their endpoints). Then the homotopy class contains a unique reduced path $\ol{\sigma}:I_r\rar G$. Moreover, for every combinatorial path $\sigma':I_l\rar G$ and for every maximal reduction process $(s_1,t_1),...,(s_m,t_m)$ for $\sigma'$, we have $l-2m=r$ and the residual path coincides with $\ol\sigma$.
\end{myprop}
\begin{proof}
Let $p:\ot G\rar G$ be the universal cover and choose a lifting $\tau:I_l\rar\ot G$ of the combinatorial path $\sigma$: we have that $\tau$ is a combinatorial path, connecting two distinct vertices $v_0=\tau(0)$ and $v_1=\tau(1)$ of $\ot G$. Let $\ol\sigma:I_r\rar G$ be any reduced path in the homotopy class of $\sigma$: then there is a unique lifting $\ol\tau:I_r\rar\ot G$ with $\ol\tau(0)=v_0$ and $\ol\tau(1)=v_1$, and this is a reduced path. But since $\ot G$ is a tree, there is a unique reduced path connecting $v_0$ and $v_1$. This means that $\ol\tau$ is uniquely determined by the homotopy class, and thus $\ol\sigma=p\circ\ol\tau$ is uniquely determined too. The conclusion follows.
\end{proof}
The following graphical representation of a reduction process will be useful. Let $\sigma:I_l\rar G$ be a combinatorial path and let $(s_1,t_1),...,(s_m,t_m)$ be a reduction process for $\sigma$. Consider $I_l$ as a subdivision of the unit interval $[0,1]\times\{0\}\subseteq\bR^2$. For each couple $(s_i,t_i)$, take a smooth path $r_i$ in the upper half-plane connecting the midpoint of $s_i$ to the midpoint of $t_i$. The paths $r_1,...,r_m$ can be taken to be pairwise disjoint, as in figure \ref{diagramreductionprocess}.

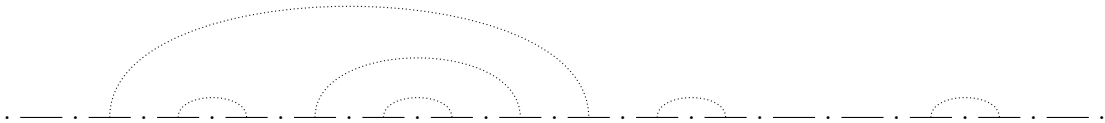
\begin{figure}[h!]
\centering
\begin{tikzpicture}[scale=0.9]
\node (0) at (0,0) {.};
\node (1) at (1,0) {.};
\node (2) at (2,0) {.};
\node (3) at (3,0) {.};
\node (4) at (4,0) {.};
\node (5) at (5,0) {.};
\node (6) at (6,0) {.};
\node (7) at (7,0) {.};
\node (8) at (8,0) {.};
\node (9) at (9,0) {.};
\node (10) at (10,0) {.};
\node (11) at (11,0) {.};
\node (12) at (12,0) {.};
\node (13) at (13,0) {.};
\node (14) at (14,0) {.};
\node (15) at (15,0) {.};
\node (16) at (16,0) {.};

\draw (0) to (1);
\draw (1) to (2);
\draw (2) to (3);
\draw (3) to (4);
\draw (4) to (5);
\draw (5) to (6);
\draw (6) to (7);
\draw (7) to (8);
\draw (8) to (9);
\draw (9) to (10);
\draw (10) to (11);
\draw (11) to (12);
\draw (12) to (13);
\draw (13) to (14);
\draw (14) to (15);
\draw (15) to (16);

\draw[out=90,in=90,looseness=0.8,densely dotted] (1.5,0) to (8.5,0);
\draw[out=90,in=90,looseness=1,densely dotted] (2.5,0) to (3.5,0);
\draw[out=90,in=90,looseness=1,densely dotted] (4.5,0) to (7.5,0);
\draw[out=90,in=90,looseness=1,densely dotted] (5.5,0) to (6.5,0);
\draw[out=90,in=90,looseness=1,densely dotted] (9.5,0) to (10.5,0);
\draw[out=90,in=90,looseness=1,densely dotted] (13.5,0) to (14.5,0);
\end{tikzpicture}
\caption{An example of a possible diagram for a reduction process.}
\label{diagramreductionprocess}
\end{figure}

\subsection{Labeled graphs}

We consider the finitely generated free group $F_n$ of rank $n$, generated by $a_1,...,a_n$. We write $\ol{a_i}=a_i^{-1}$. We denote with $R_n$ the standard $n$-rose, i.e. the graph with one vertex $*$ and $n$ oriented edges labeled $a_1,...,a_n$. The fundamental group $\pi_1(R_n,*)$ will be identified with $F_n$: the path going along the edge labeled $a_i$ (with the right orientation) corresponds to the element $a_i\in F_n$.

\begin{mydef}\label{grafo}
An \textbf{$\grafo$} is a graph $G$ together with a map $f:G\rar R_n$ sending each vertex of $G$ to the unique vertex of $R_n$, and each open edge of $G$ homeomorphically to one edge of $R_n$.
\end{mydef}

This means that every edge of $G$ is equipped with a label in $\{a_1,...,a_n\}$ and an orientation, according to which edge of $R_n$ it is mapped to; the map $f:G\rar R_n$ is called \textbf{labeling map} for $G$.

\begin{mydef}
Let $G_0,G_1$ be $\grafos$ with labeling maps $f_0,f_1$ respectively. A map $h:G_0\rar G_1$ is called \textbf{label-preserving} if $f_1\circ h=f_0$.
\end{mydef}

This means that the map $h$ sends each vertex to a vertex, and each open edge homeomorphically onto an edge with the same label and orientation. In particular $h$ is a combinatorial map.

\subsection{Core graph of a subgroup}

From the theory of covering spaces, we know that pointed covering spaces of $R_n$ are in bijection with subgroups of the fundamental group $\pi_1(R_n,*)=F_n$. Given a pointed covering space $p:(P,*)\rar(R_n,*)$, we have that the map $p_*:\pi_1(P,*)\rar F_n$ is injective, and thus $\pi_1(P,*)$ can be identified with its image $p_*(\pi_1(P,*))=H$, determining a subgroup $H\sgr F_n$. Conversely, given a subgroup $H\sgr F_n$, there is a unique pointed covering space $p:(P,*)\rar(R_n,*)$ such that $p_*(\pi_1(P,*))=H$: we define $(\cov{H},*)=(P,*)$ to be such covering space.

\begin{myrmk}
A covering space $p:(P,*)\rar(R_n,*)$ is in particular an $\grafo$.
\end{myrmk}

\begin{mydef}
Define the core graph $\core{H}$ and the pointed core graph $\bcore{H}$ to be the core and the pointed core of $(\cov{H},*)$, respectively.
\end{mydef}

Of course $\cov{H},\core{H},\bcore{H}$ are $\grafo$s, with labeling map given by the covering projection $p$, and by its restriction to the subgraphs $\core{H}$ and $\bcore{H}$ respectively. The labeling map $f:\bcore{H}\rar R_n$ gives a map $f_*:\pi_1(\bcore{H},*)\rar F_n$ which induces an isomorphism $f_*:\pi_1(\bcore{H},*)\rar H$.

%We observe that conjugate subgroups have the same core graph, but distinct pointed core graphs.

We have that $H$ is finitely generated if and only if $\core{H}$ is finite (and if and only if $\bcore{H}$ is finite). In that case, $\core{H}$ and $\bcore{H}$ can be built algorithmically from a finite set of generators for $H$, see algorithm 5.4 %on page 557
in \cite{Stallings}.

Given two finitely generated subgroups $H_1,H_2$, it is possible to algorithmically build the core graph $\core{H_1\cap H_2}$ of their intersection, see theorem 5.5 %on page 558
in \cite{Stallings}. This also allows one to prove Howson's theorem, stating that the intersection of two finitely generated subgroups of a free group is finitely generated.

\section{Equations and Stallings folding}\label{SectionStallings}

Fix $H\sgr F_n$ finitely generated and $g\in F_n$ that depends on $H$. We here introduce an efficient way of computing a set of generators for the ideal $\fI_g\sgr H*\gen{x}$ as a normal subgroup. The technique is based on the classical Stallings folding operations; the novel aspect of what we do is that we focus on the non-rank-preserving folding operations, which are the ones responsible for the generators of the ideal, and we delay them until the end of the chain of folding operations. The same technique can be used more generally to compute a set of normal generators for the kernel of any homomorphism between free groups.

\subsection{Stallings folding}

We will assume that the reader has some confidence with the classical Stallings folding operation, for which I refer to \cite{Stallings}. I briefly recall the main properties that we are going to use.

Let $G$ be a finite connected $\grafo$ and suppose there are two distinct edges $e_1,e_2$ with endpoints $v,v_1$ and $v,v_2$ respectively. Suppose that $e_1$ and $e_2$ have the same label and orientation. We can identify $v_1$ with $v_2$, and $e_1$ with $e_2$: we then get a label-preserving quotient map of graphs $q:G\rar G'$.

\begin{mydef}
The quotient map $q:G\rar G'$ is called \textbf{Stallings folding}.
\end{mydef}

Given a finite connected $\grafo$ $G$, we can successively apply folding operations to $G$ in order to get a sequence $G=G^0\rar G^1\rar...\rar G^l$. Notice that the number of edges decreases by $1$ at each step, and thus the length of any such chain is bounded (by the number of the edges of $G$). The following proposition, although not explicitly stated in \cite{Stallings}, is a well-known consequence.

\begin{myprop}\label{folding}
Let $G$ be a finite connected $\grafo$ and let $G=G^0\rar G^1\rar...\rar G^m$ be a maximal sequence of folding operations. Also, fix a basepoint $*\in G$, inducing a basepoint $*\in G^i$ for $i=0,...,l$. Then we have the following:

(i) Each such sequence has the same length $m$ and the same final graph $G^m$.

(ii) Let $f^i:G^i\rar R_n$ be the labeling map. Then the image of $f^i_*:\pi_1(G^i,*)\rar\pi_1(R_n,*)$ is the same subgroup $H\sgr F_n$ for every $i=1,...,m$.

(iii) For every $i=1,...,m$ there is a unique label-preserving map of pointed graphs $h^i:G^i\rar\cov{H}$. The image $\im{h^i}$ is the same subgraph of $\cov{H}$ for every $i=1,...,m$.

(iv) The map $h^m$ is an embedding of $G^m$ as a subgraph of $\cov{H}$ and the subgraph $h^m(G^m)$ contains $\bcore{H}$. In particular $h^m_*:\pi_1(G^m,*)\rar\pi_1(\cov{H},*)$ is an isomorphism and the map $f^m_*:\pi_1(G^m,*)\rar F_n$ is injective.
\end{myprop}

\begin{mydef}
Let $G$ be a finite connected $\grafo$. Define its \textbf{folded graph} $\fold{G}$ to be the $\grafo$ $G^m$ obtained from any maximal sequence of folding operations as in Proposition \ref{folding}.
\end{mydef}

\subsection{Rank-preserving and non-rank-preserving folding operations}

Let $G$ be an $\grafo$ and let $q:G\rar G'$ be a folding operation.

\begin{mydef}
A Stallings folding $q:G\rar G'$ is called \textbf{rank-preserving} if  it is an homotopy equivalence.
\end{mydef}

In that case, for every basepoint $*\in G$, the map $q:(G,*)\rar(G',q(*))$ is a pointed homotopy equivalence and $q_*:\pi_1(G,*)\rar\pi_1(G',q(*))$ is an isomorphism. Being rank-preserving is equivalent to the requirement that the endpoints that we are identifying are distinct (see also figure \ref{rankpreserving}).

\begin{figure}[h!]
\centering
\begin{tikzpicture}
\node (1) at (2,1) {.};
\node (2) at (1,3) {.};
\node (3) at (3,3) {.};
\draw[->] (1) to node[left]{$a$} (2);
\draw[->] (1) to node[right]{$a$} (3);

\node (4) at (5,1) {.};
\node (5) at (7,1) {.};
\draw[->] (4) to[out=115,in=65,looseness=50] node[below]{$a$} (4);
\draw[->] (4) to node[above]{$a$} (5);

\node (6) at (10,1) {.};
\node (7) at (10,3) {.};
\draw[->] (6) to[out=115,in=-115,looseness=1] node[left]{$a$} (7);
\draw[->] (6) to[out=65,in=-65,looseness=1] node[right]{$a$} (7);

\node (8) at (13,1) {.};
\draw[->] (8) to[out=140,in=110,looseness=50] node[left]{$a$} (8);
\draw[->] (8) to[out=70,in=40,looseness=50] node[right]{$a$} (8);
\end{tikzpicture}
\caption{Examples of configurations where a folding operation is possible. The two examples on the left produce rank-preserving folding operations; the two examples on the right produce non-rank-preserving folding operations.}
\label{rankpreserving}
\end{figure}
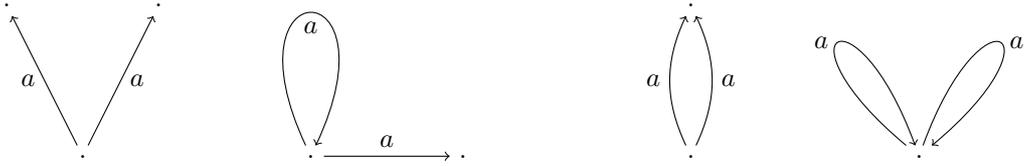

\begin{figure}[h!]
\centering
\begin{tikzpicture}
\node (1) at (0,0) {.};
\node (2) at (-1,2) {.};
\node (3) at (1,2) {.};
\node (4) at (-2,4) {.};
\node (5) at (0,4) {.};
\draw[->] (1) to node[left]{$a$} (2);
\draw[->] (1) to node[right]{$a$} (3);
\draw[->] (2) to node[left]{$b$} (4);
\draw[->] (2) to node[right]{$b$} (5);

\node (6) at (6,0) {.};
\node (7) at (5,2) {.};
\node (8) at (7,2) {.};
\node (9) at (4,4) {.};
\node (10) at (8,4) {.};
\draw[->] (6) to node[left]{$a$} (7);
\draw[->] (6) to node[right]{$a$} (8);
\draw[->] (7) to node[left]{$b$} (9);
\draw[->] (8) to node[right]{$b$} (10);
\end{tikzpicture}
\caption{Above we have two examples of graphs, and in each of them we want to perform two folding operations (one involving $a$-labeled edges, and the other involving $b$-labeled edges). In the graph on the left, we can perform the two operations in any order. In the graph on the right, we are forced to perform the operation on the $a$-labeled edges first.}
\label{commutefolding}
\end{figure}
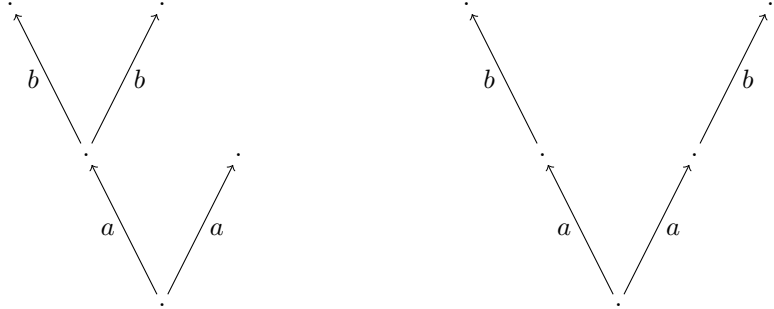

In a sequence of folding operations as in Proposition \ref{folding}, it is not always possible to change the order of the operations; see for example figure \ref{commutefolding}. Informally, we could say that certain folding operations are required before being able to perform other operations. The key observation is that the non-rank-preserving folding operations change the set of edges of $G$, but they do not change the set of vertices of $G$; as a consequence, they are not a requirement for any other operation. This can be made precise as follows.

\begin{myprop}\label{folding2}
Let $G$ be a finite connected $\grafo$. Let $G=G^0\rar G^1\rar...\rar G^k$ be a maximal sequence of rank-preserving folding operations. Let $G^k\rar G^{k+1}\rar...\rar G^m$ be a maximal sequence of folding operations for $G^k$. Also, fix a basepoint $*\in G$, inducing a basepoint $*\in G^i$ for every $i=1,...,m$. Then we have the following:

(i) Each map in the first sequence is a (pointed) homotopy equivalence; the map $G^0\rar G^k$ is an homotopy equivalence.

(ii) The second sequence only contains non-rank-preserving folding operations; the map $G^k\rar G^m$ is an isomorphism on the set of vertices.

(iii) The concatenation of the two sequences produces a folding sequence as in Proposition \ref{folding}. In particular $G^m=\fold{G}$.

(iv) The numbers $k,m$ do not depend on the chosen sequences.
\end{myprop}
\begin{myrmk}
This shows that the graph $G^k$ is essentially $\fold{G}$, but with some edge repeated two or more times (see figure \ref{gmgl}). The repeated edges (and their multiplicity) can depend on the chosen sequence of folding operations; the graph $G^k$ is not uniquely determined by $G$.
\end{myrmk}

\begin{proof}
Part (i) is trivial.

For (ii), suppose the sequence $G^k\rar...\rar G^m$ contains a rank-preserving folding operation, and let $j\ge k$ be the smallest integer such that $G^j\rar G^{j+1}$ is rank-preserving; this means that there are two edges $e_1,e_2$ in $G^j$ with an endpoint $v$ in common, the other endpoints $v_1\not=v_2$ distinct, and the same label and orientation. Let $p:G^k\rar G^j$ be the composition of the sequence of folding operations $G^k\rar...\rar G^j$: each of those operations is non-rank-preserving, and in particular it induces an isomorphism on the set of vertices. Thus we can take the vertices $p^{-1}(v)$ and $p^{-1}(v_1)\not=p^{-1}(v_2)$. Take any edge $e_3\in p^{-1}(e_1)$ and $e_4\in p^{-1}(e_2)$ and we have that in $G^k$ is it possible to fold $e_3$ and $e_4$, performing a rank-preserving folding operation. This gives a contradiction because the sequence of rank-preserving folding operations $G^0\rar...\rar G^k$ was maximal.

Part (iii) is trivial.

For part (iv), we observe the following: along the sequence $G^0\rar...\rar G^k$, at each step the number of vertices decreases by one, while along the sequence $G^k\rar...\rar G^m$ the number of vertices is preserved. Thus $k$ is equal to the number of vertices of $G$ minus the number of vertices of $\fold{G}$, regardless of the chosen sequence. By Proposition \ref{folding} the sum $m+k$ doesn't depend on the chosen sequence, and thus neither does $m$.
\end{proof}

\begin{figure}
\centering
\begin{tikzpicture}
\node (a1) at (0,0) {*};
\node (a2) at (-2,1) {.};
\node (a3) at (-2,3) {.};
\node (a4) at (-1,4) {.};
\node (a5) at (1,4) {.};
\node (a6) at (2,3) {.};
\node (a7) at (2,1) {.};
\draw (a1) to (a2);
\draw[out=90,in=270,looseness=1] (a2) to (a3);
\draw[out=110,in=250,looseness=1] (a2) to (a3);
\draw (a3) to (a4);
\draw (a4) to (a5);
\draw[out=120,in=60,looseness=25] (a5) to (a5);
\draw[out=125,in=55,looseness=25] (a5) to (a5);
\draw (a5) to (a6);
\draw[out=-115,in=115,looseness=1] (a6) to (a7);
\draw[out=-125,in=125,looseness=1] (a6) to (a7);
\draw[out=-135,in=135,looseness=1] (a6) to (a7);
\draw[out=-60,in=60,looseness=1] (a6) to (a7);
\draw[out=-50,in=50,looseness=1] (a6) to (a7);
\draw (a7) to (a1);
\draw (a4) to (a1);
\node (a15) at (-2,5) {$G^k$};

%\draw[->] (3,2) to (6.5,2);

\node (b1) at (9,0) {*};
\node (b2) at (7,1) {.};
\node (b3) at (7,3) {.};
\node (b4) at (8,4) {.};
\node (b5) at (10,4) {.};
\node (b6) at (11,3) {.};
\node (b7) at (11,1) {.};
\draw (b1) to (b2);
\draw (b2) to (b3);
\draw (b3) to (b4);
\draw (b4) to (b5);
\draw[out=120,in=60,looseness=25] (b5) to (b5);
\draw (b5) to (b6);
\draw[out=-120,in=120,looseness=1] (b6) to (b7);
\draw[out=-60,in=60,looseness=1] (b6) to (b7);
\draw (b7) to (b1);
\draw (b4) to (b1);
\node (b15) at (7,5) {$G^m$};
\end{tikzpicture}
\caption{An example of the result of the folding procedure described in Proposition \ref{folding2} (the labeling has been omitted). On the left the graph $G^k$, the result of the sequence of only rank-preserving folding operations. On the right the graph $G^m$ obtained from a sequence of both rank-preserving and non-rank-preserving operations.}\label{gmgl}
\end{figure}
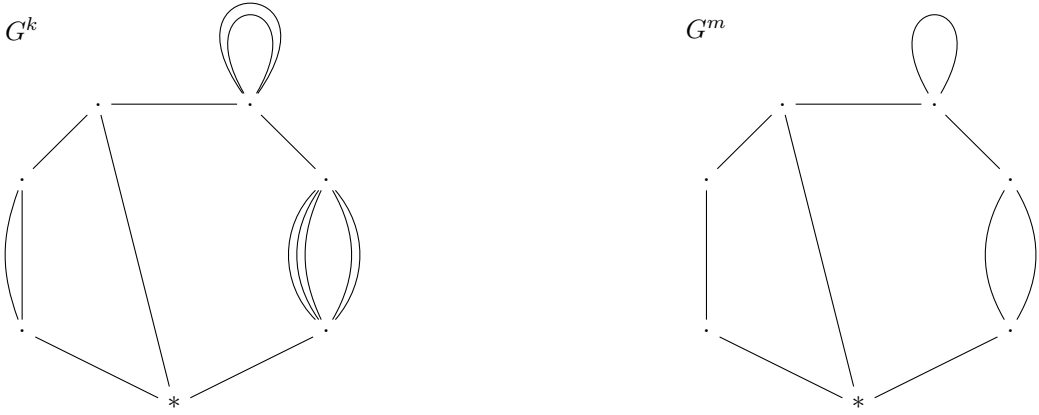

\subsection{A set of normal generators for the set of equations}\label{SubsectionG}

\begin{mythm}\label{idealfingen}
Let $H\sgr F_n$ be a finitely generated subgroup and $g\in F_n$ be an element. Then we have the following:

(i) The ideal $\fI_g\nor H*\gen{x}$ is finitely generated as a normal subgroup.

(ii) The set of generators for $\fI_g$ can be taken to be a subset of a basis for $H*\gen{x}$.

(iii) There is an algorithm that, given $H$ and $g$, computes a finite set of normal generators for $\fI_g$ which is also a subset of a basis for $H*\gen{x}$.
\end{mythm}
\begin{proof}
Let $\varphi_g:H*\gen{x}\rar F_n$ be the corresponding evaluation map, so that $\fI_g=\ker{\varphi_g}$. Let $G=\bcore{H}\vee\bcore{\gen{g}}$ be the pointed $\grafo$ obtained by identifying the basepoints of $\bcore{H}$ and $\bcore{\gen{g}}$; see figure \ref{graphG}. Let $f:G\rar R_n$ be the labeling map, inducing a map $f_*:\pi_1(G,*)\rar\pi_1(R_n,*)$ between the fundamental groups.

Let $\theta:H*\gen{x}\rar\pi_1(G,*)$ be the isomorphism sending each element of $H$ to the corresponding path in $\bcore{H}$, and the element $x$ to the path in $\bcore{\gen{g}}$ corresponding to the element $g$. It is immediate to see that $f_*\circ\theta=\varphi_g$ as maps from $H*\gen{x}$ to $\pi_1(R_n,*)=F_n$: in particular we have $\fI_g=\ker{\varphi_g}=\ker(f_*\circ\theta)=\theta^{-1}(\ker{f_*})$.

\begin{figure}[h!]
\centering
\begin{tikzpicture}[scale=1]
\node (1) at (0,0) {*};
\node (2) at (-2,-1.2) {.};
\node (3) at (-4,-1.2) {.};
\node (4) at (-6,0) {.};
\node (5) at (-2,1.2) {.};
\node (6) at (-4,1.2) {.};

\node (8) at (2,1) {.};
\node (9) at (2,-1) {.};

\draw[->] (1) to node[below]{$a$} (2);
\draw[->] (2) to node[below]{$b$} (3);
\draw[->] (3) to node[below]{$a$} (4);
\draw[->] (1) to node[above]{$b$} (5);
\draw[->] (5) to node[above]{$b$} (6);
\draw[->] (4) to node[above]{$a$} (6);
\draw[->] (4) to node[above]{$b$} (1);

\draw[->] (1) to node[above]{$b$} (8);
\draw[->] (8) to node[right]{$b$} (9);
\draw[->] (1) to node[below]{$a$} (9);

\node (a) at (0,-1.6) {$G$};

%\draw[out=100,in=90,looseness=2.1,loosely dotted] (1) to (-4.3,0);
%\draw[out=-90,in=-100,looseness=2.1,loosely dotted] (-4.3,0) to (1);
%\node (H) at (-5.3,0) {$\bcore{H}$};
%\draw[out=80,in=90,looseness=2.2,loosely dotted] (1) to (3,0);
%\draw[out=-90,in=-80,looseness=2.2,loosely dotted] (3,0) to (1);
%\node (g) at (4,0) {$\bcore{\gen{g}}$};
\end{tikzpicture}
\caption{In the picture we can see the graph $G$. Here $F_2=\gen{a,b}$ and $H=\gen{h_1,h_2}$ with $h_1=b^2\ol{a}b$ and $h_2=abab$. On the left of the basepoint we have the graph $\bcore{H}$. On the right of the basepoint we have the graph $\bcore{\gen{g}}$ where $g=b^2\ol{a}$.}
\label{graphG}
\end{figure}
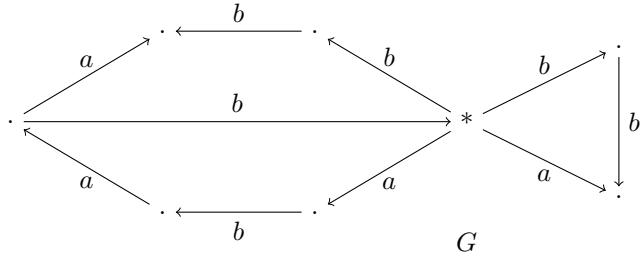

Let $G=G^0\rar...\rar G^k$ be a maximal sequence of rank-preserving folding operations and let $G^k\rar...\rar G^m$ be a maximal sequence of folding operations for $G^k$, as in Proposition \ref{folding2}. The basepoint $*\in G$ induces a basepoint $*\in G^i$ for every $i=0,...,m$. Let $p:G\rar G^k$ and $q:G^k\rar G^m$ be the quotient maps given by the sequences of foldings. Let $R_n$ be the $n$-rose with edges labeled $a_1,...,a_n$ and let $f:G\rar R_n$ and $f^k:G^k\rar R_n$ and $f^m:G^m\rar R_n$ be the labeling maps.

By Proposition \ref{folding2} we have that $p_*:\pi_1(G,*)\rar\pi_1(G^k,*)$ is an homotopy equivalence; by Proposition \ref{folding} we have that $f^m_*:\pi_1(G^m,*)\rar\pi_1(R_n,*)=F_n$ is injective. Since diagram \ref{maps} commutes, we have $\ker{\varphi_g}=\ker(f_*\circ\theta)=\ker(f^m_*\circ q_*\circ p_*\circ\theta)=\theta^{-1}(p_*^{-1}(\ker{q_*}))$. We now show that $\ker{q_*}$ is quite easy to compute, and the maps $p_*^{-1}$ and $\theta^{-1}$ can be made explicit too.

\begin{figure}[h!]
\centering
\begin{tikzpicture}
\node (0) at (0,2) {$H*\gen{x}$};
\node (1) at (3,2) {$\pi_1(G,*)$};
\node (2) at (6,2) {$\pi_1(G^k,*)$};
\node (3) at (9,2) {$\pi_1(G^m,*)$};
\node (4) at (3,0) {$F_n=\pi_1(R_n,*)$};

\draw[->] (0) to node[above]{$\theta$} (1);
\draw[->] (1) to node[above]{$p_*$} (2);
\draw[->] (2) to node[above]{$q_*$} (3);
\draw[->] (0) to node[left]{$\varphi_g$} (1.8,0.3);
\draw[->] (1) to node[left]{$f_*$} (3.4,0.3);
\draw[->] (2) to node[left]{$f^k_*$} (3.85,0.3);
\draw[->] (3) to (4.3,0.3);
\node (9) at (5.6,1) {$f^m_*$};
\end{tikzpicture}
\caption{The diagram commutes.}\label{maps}
\end{figure}
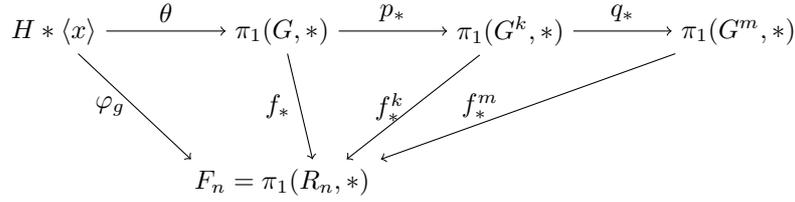

Let $T$ be a maximal tree for $G^k$. Let $e_1,...,e_{r}$ be the list of edges in $G^k\setminus T$ (each with its orientation coming from the labeling): these give a basis $\sigma_1,...,\sigma_{r}$ for the fundamental group $\pi_1(G^k,*)$, as defined in Proposition \ref{fundamentalgroupgraph}.
%; notice that, since $p_*$ is an homotopy equivalence, we must have $r'=r$.
By Proposition \ref{folding2}, the map $q:G^k\rar G^m$ is an isomorphism on the set of vertices (see figure \ref{gmgl}), and thus $\restr{q}{T}$ is an homeomorphism and $q(T)$ is a maximal tree for $G^m$. Let $d_1,...,d_s$ be the list of edges in $G^m\setminus q(T)$ (each with its orientation coming from the labeling map): these give a basis $\tau_1,...,\tau_s$ for the fundamental group $\pi_1(G^m,*)$, according to Proposition \ref{fundamentalgroupgraph}.

The map $q_*:\pi_1(G^k,*)\rar\pi_1(G^m,*)$ is now very easy to describe: we have

$$q_*(\sigma_i)=\begin{cases}
1&\text{if }q(e_i)\in q(T)\\
\tau_j&\text{if }q(e_i)=d_j\in G^m\setminus q(T)
\end{cases}$$

For each $d_j\in G^m\setminus q(T)$ fix an index $i(j)$ such that $q(e_{i(j)})=d_j$. Then we can define the set
$$N=\{\sigma_i : q_*(\sigma_i)=1\}\cup\{\sigma_{i'}\sigma_{i(j)}^{-1} : q_*(\sigma_{i'})=d_j\text{ and }i'\not=i(j)\}$$
and we observe that $N$ is a set of normal generators for $\ker{q_*}$, and it also has the additional property of being a subset of a basis for $\pi_1(G^k,*)$.

We observe that the inverse $p_*^{-1}$ can be made explicit as follows. Each folding operation in the chain $G=G^0\rar...\rar G^k$ is a pointed homotopy equivalence, and it is easy to produce homotopy inverses $\alpha^i:G^i\rar G^{i-1}$ for $i=1,...,k$. We now take the composition $\alpha=\alpha^1\circ...\circ\alpha^k:G^k\rar G$ and we observe that the map $\alpha_*:\pi_1(G^k,*)\rar\pi_1(G,*)$ is exactly the desired inverse $\alpha_*=p_*^{-1}$.

Finally, the map $\theta^{-1}$ works as follows: we take a path $\gamma:I_l\rar G$ with $\gamma(0)=\gamma(1)=*$, we take its reduction $\ol\gamma$, and we write down the word that we read while going along $\ol\gamma$; whenever we cross $\bcore{x}$ we write $x$ or $\ol{x}$ instead of the labels of the edges of $\bcore{x}$. The result of this process is exactly the element $\theta^{-1}([\gamma])\in H*\gen{x}$.
\end{proof}

\begin{myrmk}
Notice that with the above argument, we are able to produce a basis $c_1,...,c_{r+1}$ for $H*\gen{x}$, where $c_i=\theta^{-1}(p_*^{-1}(\sigma_i))$, such that each of $h_1,...,h_r,x$ written as a reduced word in $c_1,...,c_{r+1}$ has at most the same length as $h_1,...,h_r,g$ written as a reduced word in $a_1,...,a_n$, respectively; moreover, the ideal $\fI_g$ is generated (as normal subgroup) by words in $c_1,...,c_{r+1}$ of length at most $2$.
\end{myrmk}

\begin{myrmk}
Consider the subgroup $\gen{H,g}\sgr F_n$. We can compute a basis for $\gen{H,g}$, and we can then use $\gen{H,g}$ as ambient group instead of $F_n$ itself; this doesn't change the kernel $\ker{\varphi_g}$. In other words, we can assume that the graph $G^m$ that we obtain at the end of the folding process is exactly the rose $R_n$, and that the graph $G^k$ is a rose too, but with some label repeated more than once on the petals. This assumption makes the computations easier.
\end{myrmk}

\section{The minimum degree of an equation}\label{SectionMain}

In this section, we work with a fixed finitely generated subgroup $H\sgr F_n$ and with a fixed element $g\in F_n$ such that $g$ depends on $H$. With the same notation as in the proof of Theorem \ref{idealfingen}, we consider the $\grafo$ $G=\bcore{H}\vee\bcore{\gen{g}}$ with labeling map $f:G\rar R_n$, inducing a map $f_*:\pi_1(G,*)\rar F_n$ of fundamental groups.  We consider the isomorphism $\theta:H*\gen{x}\rar\pi_1(G,*)$ as defined in the proof of Theorem \ref{idealfingen}.

\begin{mydef}\label{defcorrpath}
Let $w\in H*\gen{x}$ be a non-trivial equation. Define the \textbf{corresponding path} $\sigma:I_l\rar G$ as the unique reduced path in the homotopy class $\theta(w)$ (see Proposition \ref{reducedpath}).
\end{mydef}

\begin{mydef}\label{defcorrequation}
Let $\sigma:I_l\rar G$ be a reduced path with $\sigma(0)=\sigma(1)=*$. Define the \textbf{corresponding equation} $w\in H*\gen{x}$ as $w=\theta^{-1}([\sigma])$.
\end{mydef}

The two above definitions give a bijection between non-trivial equations $w\in H*\gen{x}$ and reduced paths $\sigma:I_l\rar G$. Cyclically reduced paths correspond to \textit{cyclically reduced equations}, i.e. equations $w\in H*\gen{x}$ such that, when we write $w$ as a reduced word in the letters $a_1,...,a_n,x$, the word is also cyclically reduced.

The aim of this section is to prove the following theorem:

\begin{mythm}\label{main}
Let $L$ be the number of edges of the graph $G$ and let $d_{min}$ be the minimum possible degree for a non-trivial equation in $\fI_g$. Then there is a non-trivial equation $w\in\fI_g$ of degree $d_{min}$ and such that the corresponding path $\sigma:I_l\rar G$ has length $l\le16L^2d_{min}$.
\end{mythm}

\subsection{Innermost cancellations}

The degree of a cyclically reduced equation can be computed by looking at how many times we cross the edges of $\core{\gen{g}}$.

\begin{mylemma}\label{degree}
Let $\sigma:I_l\rar G$ be a cyclically reduced path. Let $e$ be any edge of $G$ that belongs to the subgraph $\core{\gen{g}}$. Then the degree of the equation $w$ corresponding to $\sigma$ coincides with the number of times $\sigma$ crosses the edge $e$ (in either direction).
\end{mylemma}
\begin{proof}
Write the equation $w$ as a cyclically reduced word $c_1x^{\alpha_1}c_2x^{\alpha_2}...c_rx^{\alpha_r}c_{r+1}$ with $\alpha_1,...,\alpha_r\in\bZ\setminus\{0\}$ and $c_1,...,c_{r+1}\in H$. Then in the graph $G$ we have that $\theta(w)=\theta(c_1)\cdot\theta(x^{\alpha_1})\cdot...\cdot\theta(x^{\alpha_r})\cdot\theta(c_{r+1})$, where the $\cdot$ symbol denotes the concatenation of paths (without any homotopy). It is immediate to see that $\theta(x^{\alpha_i})$ crosses each edge of $\core{\gen{g}}$ exactly $\abs{\alpha_i}$ times, for $i=1,...,r$, and that $\theta(c_i)$ is contained in $\bcore{H}$ and thus it doesn't cross any edge of $\core{\gen{g}}$. The conclusion follows.
\end{proof}

%
%Consider a folding sequence for the graph $G$, as in Proposition \ref{folding}, and let $p:G\rar\fold{G}$ be the (label-preserving and basepoint-preserving) map given by the sequence of folding operations; let $f:G\rar R_n$ and $f':\fold{G}\rar R_n$ be the labeling maps. Let $\theta:H*\gen{x}\rar\pi_1(G,*)$ be the isomorphism sending each element of $H$ to the corresponding path in $\bcore{H}$, and the element $x$ to the path in $\bcore{\gen{g}}$ corresponding to the element $g$. Then the diagram of figure \ref{diagram} commutes. But by Lemma \ref{injective} we have that $f'_*$ is injective: in particular we have $\ker{\varphi_g}=\ker(f'_*\circ p_*\circ\theta)=\ker(p_*\circ\theta)=\theta^{-1}(\ker{p_*})$.
%
%The map $\theta$ is a bijection between elements $w\in H*\gen{x}$ and homotopy classes $[\sigma]$ of paths in $G$ from $*$ to itself; also, $\theta$ induces a bijection between equations $w\in\fI_g$ and homotopy classes $[\sigma]$ such that $p\circ\sigma$ is homotopically trivial (relative to the endpoints).
%
%\begin{figure}[h!]
%\centering
%\begin{tikzpicture}[scale=1]
%\node (1) at (-4,0) {$H*\gen{x}$};
%\node (2) at (0,0) {$\pi_1(G,*)$};
%\node (3) at (4,0) {$\pi_1(\fold{G},*)$};
%\node (4) at (0,-2) {$F_n=\pi_1(R_n,*)$};
%
%\draw[->] (1) to node[above]{$\theta$} (2);
%\draw[->] (2) to node[above]{$p_*$} (3);
%\draw[->] (1) to node[below]{$\varphi_g$} (4);
%\draw[->] (2) to node[left]{$f_*$} (4);
%\draw[->] (3) to node[below]{$f'_*$} (4);
%\end{tikzpicture}
%\caption{The diagram of maps commutes.}
%\label{diagram}
%\end{figure}

Non-trivial equations $w\in\fI_g$ with $g$ as a solution correspond to reduced paths $\sigma:I_l\rar G$ such that $f\circ\sigma$ is homotopycally trivial (relative to the endpoints). In this case we can take a maximal reduction process $(s_1,t_1),...,(s_{l/2},t_{l/2})$ for $f\circ\sigma$. The following two lemmas, which will be of fundamental importance in what follows, tell us that the degree of an equation is closely related to the number of innermost cancellations.

\begin{mydef}
Let $\sigma:I_l\rar G$ be a reduced path with $\sigma(0)=\sigma(1)=*$ and such that $f\circ\sigma$ is homotopycally trivial (relative to the endpoints); let $(s_1,t_1),...,(s_{l/2},t_{l/2})$ be a maximal reduction process for $f\circ\sigma$. A couple $(s_i,t_i)$ is called \textbf{innermost cancellation} if $s_i$ and $t_i$ are adjacent on the interval $I_l$.
\end{mydef}

\begin{mylemma}\label{innermostcouple}
Let $\sigma:I_l\rar G$ be a reduced path with $\sigma(0)=\sigma(1)=*$ and such that $f\circ\sigma$ is homotopically trivial (relative to its endpoints); let $(s_1,t_1),...,(s_{l/2},t_{l/2})$ be a maximal reduction process for $f\circ\sigma$. Let $(s_i,t_i)$ be an innermost cancellation. Then, among $\sigma(s_i)$ and $\sigma(t_i)$, one is an edge of $\bcore{H}$ and the other is an edge of $\bcore{\gen{g}}$ (and the vertex between them is the basepoint).
\end{mylemma}
\begin{proof}
Since $(s_i,t_i)$ is a couple of a reduction process for $f\circ\sigma$, we have that $(f\circ\sigma)(s_i)$ and $(f\circ\sigma)(t_i)$ are the same edge of $\fold{G}$ but with opposite orientations. In particular $(f\circ\sigma)(s_i)$ and $(f\circ\sigma)(t_i)$ have the same label and opposite orientations, and, since $p$ is label-preserving, the two edges $\sigma(s_i)$ and $\sigma(t_i)$ have the same label and opposite orientations too. Observe that $\sigma(s_i)$ and $\sigma(t_i)$ are adjacent but distinct, since $\sigma$ is a reduced path. This means that $\sigma(s_i)$ and $\sigma(t_i)$ can't both belong to $\bcore{H}$ (because it is folded), and can't both belong to $\bcore{\gen{g}}$ (because it is folded too). Thus one of them has to belong to $\bcore{H}$ and the other to $\bcore{\gen{g}}$, and the conclusion follows.
\end{proof}

\begin{mylemma}\label{innermostbounded}
Let $\sigma:I_l\rar G$ be a cyclically reduced path with $\sigma(0)=\sigma(1)=*$ and such that $f\circ\sigma$ is homotopically trivial (relative to its endpoints); let $(s_1,t_1),...,(s_{l/2},t_{l/2})$ be a maximal cancellation process for $f\circ\sigma$. Suppose the equation $w\in\fI_g$ corresponding to $\sigma$ has degree $d$. Then the reduction process contains at most $2d$ innermost cancellations.
\end{mylemma}
\begin{proof}
As in the proof of Lemma \ref{degree}, write the equation $w$ as a cyclically reduced word
$$c_1x^{\alpha_1}c_2x^{\alpha_2}...c_rx^{\alpha_r}c_{r+1}$$
with $\alpha_1,...,\alpha_r\in\bZ\setminus\{0\}$ and $c_1,...,c_{r+1}\in H$. In the graph $G$ we have $\sigma=\theta(w)=\theta(c_1)\cdot\theta(x^{\alpha_1})\cdot...\cdot\theta(x^{\alpha_r})\cdot\theta(c_{r+1})$, where the $\cdot$ symbol denotes the concatenation of paths (without any homotopy). We see that $w$ has degree $d=\abs{\alpha_1}+...+\abs{\alpha_r}\ge r$ and, using Lemma \ref{innermostcouple}, that the path $\sigma$ contains at most $2r$ innermost cancellations. The conclusion follows.
\end{proof}

%The above Lemmas \ref{degree} and \ref{innermostcouple} help us relate the degree of an equation with the number of innermost cancellations in the reduction process. Let $\sigma:I_l\rar G$ be a reduced path such that $f\circ\sigma$ is homotopycally trivial, together with a maximal reduction process $(s_1,t_1),...,(s_{l/2},t_{l/2})$ for $f\circ\sigma$. Let $d$ be the degree of $\sigma$ and let $c$ be the number of innermost cancellations, i.e. of couples $(s_i,t_i)$ such that $s_i$ and $t_i$ are adjacent on $I_l$. By Lemma \ref{innermostcouple}, each such couple contains an edge of $\bcore{\gen{g}}$ adjacent to the basepoint, and there are exactly two such edges (because we assumed that $g$ is cyclically reduced): at least one of those two edges has to be crossed at least $c/2$ times, and thus by Lemma \ref{degree} we have $c\le 2d$.

\subsection{Parallel cancellation}

In this subsection we introduce the parallel cancellation moves, which allow us to produce a shorter equation from a longer one. We give a characterization of which parallel cancellation moves preserve the degree of the equation. Recall that $I_l$ is the unit interval $[0,1]$ subdivided into $l$ segments, and recall that for an edge $s$ of $I_l$, we denote with $\omeno{s},\opiu{s}$ the endpoints of $s$, ordered on the interval in such a way that $\omeno{s}<\opiu{s}$.

\begin{mydef}\label{parallel}
Let $\sigma:I_l\rar G$ be a reduced path with $\sigma(0)=\sigma(1)=*$ and such that $f\circ\sigma$ is homotopically trivial; let $(s_1,t_1),...,(s_{l/2},t_{l/2})$ be a maximal reduction process for $f\circ\sigma$. We say that two couples $(s_\alpha,t_\alpha),(s_\beta,t_\beta)$ with $\alpha<\beta$ are \textbf{$\parallele$} if they satisfy the following conditions:

(i) The edges $s_\beta,s_\alpha,t_\alpha,t_\beta$ appear in this order on $I_l$.

(ii) The map $\sigma$ sends $s_\alpha,s_\beta$ to the same edge of $G$ crossed with the same orientation.

(iii) The map $\sigma$ sends $t_\alpha,t_\beta$ to the same edge of $G$ crossed with the same orientation.
\end{mydef}

The reason behind the definition of $\parallele$ couples is that they allow us to perform a cancellation move, which I now describe, that will be of fundamental importance in the proof of the main theorem. Let $(s_\alpha,t_\alpha),(s_\beta,t_\beta)$ be two parallel couples; we take the subgraph of $I_l$ given by the interval $[\omeno{s_\beta},\omeno{s_\alpha}]$ and collapse it to a point; we also take the subgraph of $I_l$ given by the interval $[\opiu{t_\alpha},\opiu{t_\beta}]$ and collapse it to a point (see figure \ref{cancellationmove}). We obtain a graph isomorphic to $I_{l'}$, and notice that $2\le l'\le l-2$ (because there are at least two edges that get collapsed, namely $s_\beta$ and $t_\beta$, and two that don't get collapsed, namely $s_\alpha$ and $t_\alpha$). We can define a map $\sigma':I_{l'}\rar G$ which is equal to $\sigma$, except on the collapsed interval $[\omeno{s_\beta},\omeno{s_\alpha}]$, where we set it equal to $\sigma(\omeno{s_\beta})=\sigma(\omeno{s_\alpha})$, and except on the collapsed interval $[\opiu{t_\alpha},\opiu{t_\beta}]$, where we set it equal to $\sigma(\opiu{t_\alpha})=\sigma(\opiu{t_\beta})$. This gives a well-defined combinatorial path $\sigma':I_{l'}\rar G$.

\begin{mylemma}[Parallel cancellation]\label{parallelcancellation}
Let $\sigma:I_l\rar G$ be a reduced path with $\sigma(0)=\sigma(1)=*$ and such that $f\circ\sigma$ is homotopically trivial; let $(s_1,t_1),...,(s_{l/2},t_{l/2})$ be a maximal reduction process for $f\circ\sigma$. Suppose for some $1\le\alpha<\beta\le l/2$ the couples $(s_\alpha,t_\alpha),(s_\beta,t_\beta)$ are $\parallele$ and define the map $\sigma':I_{l'}\rar G$ as above. Then $\sigma'$ is a reduced path with $\sigma'(0)=\sigma'(1)=*$ and $f\circ\sigma'$ is homotopically trivial. A maximal reduction process for $f\circ\sigma'$ can be obtained from $(s_1,t_1),...,(s_{l/2},t_{l/2})$ by removing the couples containing edges which get collapsed in the definition of $I_{l'}$.
\end{mylemma}
\begin{proof}
Consider the map $\sigma':I_{l'}\rar G$ and we want to show that it is reduced. At the vertices in the interval $[0,\omeno{s_\beta})$, the local injectivity of $\sigma'$ immediately follows from the local injectivity of $\sigma$; the same holds for the vertices in the intervals $[\opiu{s_\alpha},\omeno{t_\alpha}]$ and $(\opiu{t_\beta},1]$. For the vertex of $I_{l'}$ corresponding to the collapsed interval $[\omeno{s_\beta},\omeno{s_\alpha}]$, the local injectivity of $\sigma'$ follows from the local injectivity of $\sigma$ at $\omeno{s_\beta}$ (and here we use the hypothesis that $\sigma$ sends $s_\beta$ and $s_\alpha$ to the same edge of $G$, crossed with the same orientation). Similarly, for the vertex of $I_{l'}$ corresponding to the collapsed interval $[\opiu{t_\alpha},\opiu{t_\beta}]$, the local injectivity of $\sigma'$ follows from the local injectivity of $\sigma$ at $\opiu{t_\beta}$. This shows that $\sigma'$ is reduced.

It is easy to see, using Lemma \ref{lemmino1}, that for every $1\le i\le m$ we have that either both or none of $s_i,t_i$ is collapsed to a point. Consider the sequence $(s_1,t_1),...,(s_{l/2},t_{l/2})$ and we remove the couple of edges that get collapsed (preserving the order of the other couples): the remaining couples contain the edges of $I_{l'}$, each appearing exactly once. We thus get a sequence $(q_1,r_1),...,(q_{l'/2},r_{l'/2})$ of couples of edges of $I_{l'}$. We want to prove that this is a reduction process for $\sigma'$ (the thesis then immediately follows).

We take a couple $(q_i,r_i)$ for some $1\le i\le l'/2$ and we want to prove that, if in $I_{l'}$ we collapse each of $q_1,r_1,...,q_{i-1},r_{i-1}$ to a point, the edges $q_i,r_i$ become adjacent. We have $(q_i,r_i)=(s_j,t_j)$ for some $1\le j\le l/2$. But In the sequence $(s_1,t_1),...,(s_{l/2},t_{l/2})$ we have that, collapsing each of $(s_1,t_1),...,(s_{j-1},t_{j-1})$, the edges $s_j,t_j$ become adjacent; of the couples $(s_1,t_1),...,(s_{j-1},t_{j-1})$, some get collapsed when passing from $I_l$ to $I_{l'}$, and the others are exactly the couples $(q_1,r_1),...,(q_{i-1},r_{i-1})$; thus, if we collapse $(q_1,r_1),...,(q_{i-1},r_{i-1})$ too, the two edges $q_i$ and $r_i$ become adjacent, as desired.

Finally, take a couple $(q_i,r_i)$ for some $1\le i\le l'/2$, and we want to prove that $f\circ\sigma'$ sends $q_i$ and $r_i$ to the same edge of $\fold{G}$ crossed with opposite orientation. But $(q_i,r_i)=(s_j,t_j)$ for some $1\le j\le l/2$, and $f\circ\sigma$ sends $s_j$ and $t_j$ to the same edge of $\fold{G}$ crossed with opposite orientation. Since $\sigma'$ is defined to coincide with $\sigma$ on the edges $q_i=s_j$ and $r_i=t_j$, we have that $f\circ\sigma'$ sends $q_i$ and $r_i$ to the same edge of $\fold{G}$ crossed with opposite orientation.

Thus $(q_1,r_1),...,(q_{l'/2},r_{l'/2})$ is a maximal reduction process for $f\circ\sigma'$, as desired.
\end{proof}

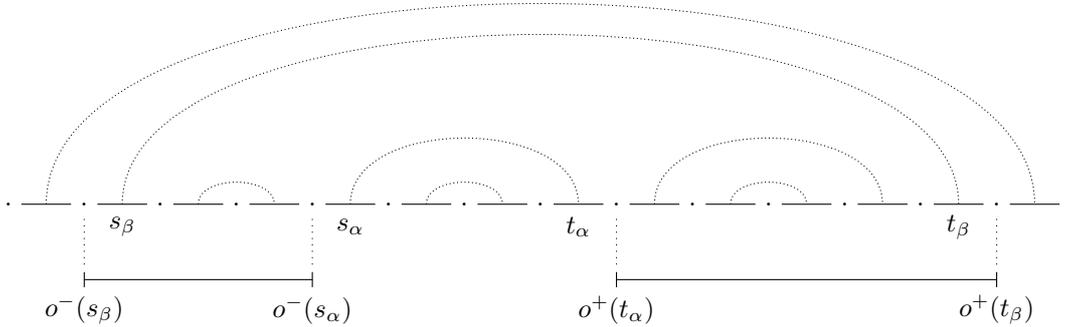
\begin{figure}[h!]
\centering
\begin{tikzpicture}
\node (0) at (0,0) {.};
\node (1) at (1,0) {.};
\node (2) at (2,0) {.};
\node (3) at (3,0) {.};
\node (4) at (4,0) {.};
\node (5) at (5,0) {.};
\node (6) at (6,0) {.};
\node (7) at (7,0) {.};
\node (8) at (8,0) {.};
\node (9) at (9,0) {.};
\node (10) at (10,0) {.};
\node (11) at (11,0) {.};
\node (12) at (12,0) {.};
\node (13) at (13,0) {.};
\node (14) at (14,0) {.};

\draw (0) to (1);
\draw (1) to (2);
\draw (2) to (3);
\draw (3) to (4);
\draw (4) to (5);
\draw (5) to (6);
\draw (6) to (7);
\draw (7) to (8);
\draw (8) to (9);
\draw (9) to (10);
\draw (10) to (11);
\draw (11) to (12);
\draw (12) to (13);
\draw (13) to (14);

\draw[out=90,in=90,looseness=0.7,densely dotted] (0.5,0) to (13.5,0);
\draw[out=90,in=90,looseness=0.7,densely dotted] (1.5,0) to (12.5,0);
\draw[out=90,in=90,looseness=1,densely dotted] (2.5,0) to (3.5,0);
\draw[out=90,in=90,looseness=1,densely dotted] (4.5,0) to (7.5,0);
\draw[out=90,in=90,looseness=1,densely dotted] (5.5,0) to (6.5,0);
\draw[out=90,in=90,looseness=1,densely dotted] (8.5,0) to (11.5,0);
\draw[out=90,in=90,looseness=1,densely dotted] (9.5,0) to (10.5,0);

\node (sb) at (1.5,-0.3) {$s_\beta$};
\node (sa) at (4.5,-0.3) {$s_\alpha$};
\node (ta) at (7.5,-0.3) {$t_\alpha$};
\node (tb) at (12.5,-0.3) {$t_\beta$};

\draw (1,-0.9) to (1,-1.1);
\draw (4,-0.9) to (4,-1.1);
\draw (1,-1) to (4,-1);
\node (sb1) at (1,-1.4) {$\omeno{s_\beta}$};
\node (sa1) at (4,-1.4) {$\omeno{s_\alpha}$};

\draw (8,-0.9) to (8,-1.1);
\draw (13,-0.9) to (13,-1.1);
\draw (8,-1) to (13,-1);
\node (ta1) at (8,-1.4) {$\opiu{t_\alpha}$};
\node (tb1) at (13,-1.4) {$\opiu{t_\beta}$};

\draw[dotted] (1,-0.8) to (1);
\draw[dotted] (4,-0.8) to (4);
\draw[dotted] (8,-0.8) to (8);
\draw[dotted] (13,-0.8) to (13);
\end{tikzpicture}
\caption{An example of a diagram for a maximal reduction process. The cancellation move collapses two intervals, which are painted below the interval, i.e. $[\omeno{s_\beta},\omeno{s_\alpha}]$ and $[\opiu{t_\alpha},\opiu{t_\beta}]$.}
\label{cancellationmove}
\end{figure}

The following two lemmas give us information about how the degree of an equation changes when we perform a parallel cancellation move on the corresponding path.

\begin{mylemma}\label{degreedecreases}
Let $\sigma:I_l\rar G$ and $\sigma':I_{l'}\rar G$ be cyclically reduced paths with $\sigma(0)=\sigma(1)=\sigma'(0)=\sigma'(1)=*$ and suppose $\sigma'$ is obtained from $\sigma$ by means of a cancellation move as described in Lemma \ref{parallelcancellation}. Then the degrees $d,d'$ of the corresponding equations $w,w'$ satisfy $d'\le d$.
\end{mylemma}
\begin{proof}
Fix an edge $e$ of $G$ belonging to $\core{\gen{g}}$ and apply Lemma \ref{degree}: the domain of $\sigma'$ is the domain of $\sigma$ with some edges collapsed, and thus the number of times $\sigma'$ crosses the edge $e$ is lesser or equal than the number of times $\sigma$ does.
\end{proof}

\begin{mydef}\label{defdegpres}
Let $\sigma:I_l\rar G$ and $\sigma':I_{l'}\rar G$ be reduced paths with $\sigma(0)=\sigma(1)=\sigma'(0)=\sigma'(1)=*$ and suppose $\sigma'$ is obtained from $\sigma$ by means of a cancellation move as described in Lemma \ref{parallelcancellation}. We say that the parallel cancellation move is \textbf{degree-preserving} if the two equations $w,w'$ corresponding to the paths $\sigma,\sigma'$ have the same degree $d=d'$.
\end{mydef}

\begin{mylemma}\label{degreepreserving}
Let $\sigma:I_l\rar G$ be a cyclically reduced path with $\sigma(0)=\sigma(1)=*$ and such that $f\circ\sigma$ is homotopically trivial; let $(s_1,t_1),...,(s_{l/2},t_{l/2})$ be a maximal reduction process for $f\circ\sigma$. Suppose there are two parallel couples $(s_\alpha,t_\alpha),(s_\beta,t_\beta)$ with $1\le\alpha<\beta\le l/2$, and let $\sigma':I_{l'}\rar G$ be the reduced path obtained with the cancellation move described in Lemma \ref{parallelcancellation}. If the cancellation move is degree-preserving, then the images by $\sigma$ of the two intervals $[\omeno{s_\beta},\omeno{s_\alpha}]$ and $[\opiu{t_\alpha},\opiu{t_\beta}]$ are contained in $\core{H}$; moreover, the two paths $\restr{f\circ\sigma}{[\omeno{s_\beta},\omeno{s_\alpha}]}$ and $\restr{f\circ\sigma}{[\opiu{t_\alpha},\opiu{t_\beta}]}$ are the same, but walked in reverse direction.
\end{mylemma}
\begin{proof}
By hypothesis, $\sigma(s_\beta)$ and $\sigma(s_\alpha)$ are the same edge $e$ of $G$ crossed with the same orientation. Suppose first that $e$ doesn't belong to $\core{H}$. Then any reduced path that starts and ends with $e$ has to cross all the edges of $\core{\gen{g}}$. Thus, by Lemma \ref{degree}, when we collapse the interval $[\omeno{s_\beta},\omeno{s_\alpha}]$ the degree strictly decreases.

Suppose now that $e$ belongs to $\core{H}$. Suppose there is an edge $s$ in the interval $[\omeno{s_\beta},\omeno{s_\alpha}]$ such that $\sigma(s)$ doesn't belong to $\core{H}$. If $\sigma(s)$ belongs to $\core{\gen{g}}$, then, by Lemma \ref{degree}, when we collapse the interval $[\omeno{s_\beta},\omeno{s_\alpha}]$ the degree strictly decreases. If $\sigma(s)$ doesn't belong to $\core{H}$ nor to $\core{\gen{g}}$, then at least one of the paths $\restr{\sigma}{[\omeno{s_\beta},\omeno{s}]}$ and $\restr{\sigma}{[\opiu{s},\omeno{s_\alpha}]}$ crosses all the edges of $\bcore{\gen{g}}$; in particular, by Lemma \ref{degree}, when collapsing the interval $[\omeno{s_\beta},\omeno{s_\alpha}]$ the degree strictly decreases. Thus the only possibility is that the path $\restr{\sigma}{[\omeno{s_\beta},\omeno{s_\alpha}]}$ is contained in $\core{H}$. Similarly, we obtain that the path $\restr{\sigma}{[\opiu{t_\alpha},\opiu{t_\beta}]}$ is contained in $\core{H}$ too. This proves the first part of the lemma.

For the second part, suppose the interval $[\omeno{s_\beta},\omeno{s_\alpha}]$ contains the two edges $s_i,t_i$ for some couple $(s_i,t_i)$ of our reduction process. Then the interval $[\omeno{s_\beta},\omeno{s_\alpha}]$ has to contain an innermost couple $(s_j,t_j)$ of our reduction process, i.e. a couple with $s_j,t_j$ adjacent on $I_l$. But by Lemma \ref{innermostcouple}, at least one of the edges $\sigma(s_j),\sigma(t_j)$ has to belong to $\bcore{\gen{g}}$, which is a contradiction with our assumptions. Thus the interval $[\omeno{s_\beta},\omeno{s_\alpha}]$ does not contain both $s_i,t_i$ for any couple $(s_i,t_i)$ of our reduction process. The same holds for the interval $[\opiu{t_\alpha},\opiu{t_\beta}]$.

Now, using Lemma \ref{lemmino1}, it is easy to see that the reduction process has to pair up the edges of the interval $[\omeno{s_\beta},\omeno{s_\alpha}]$ with the edges of the interval $[\opiu{t_\alpha},\opiu{t_\beta}]$, and the pairing has to be done in decreasing order. It follows that $\restr{f\circ\sigma}{[\omeno{s_\beta},\omeno{s_\alpha}]}$ and $\restr{f\circ\sigma}{[\opiu{t_\alpha},\opiu{t_\beta}]}$ are the same path, walked in reverse directions.
\end{proof}

\subsection{The minimum possible degree for a non-trivial equation}

Let $L$ be the number of edges of the graph $G$.

\begin{myprop}\label{findingparallels}
Let $w\in\fI_g$ be a cyclically reduced equation of degree $d$ and let $\sigma:I_l\rar G$ be the corresponding path; let $(s_1,t_1),...,(s_{l/2},t_{l/2})$ be any maximal reduction process for $f\circ\sigma$. Suppose $l>16L^2d$. Then the reduction process contains two $\parallele$ couples $(s_\alpha,t_\alpha),(s_\beta,t_\beta)$.
\end{myprop}
\begin{proof}
To each couple $(s_i,t_i)$ we associate the quadruple of edges $(\sigma(s_i),\epsilon,\sigma(t_i),\delta)$ where $\sigma(s_i),\sigma(t_i)$ are edges of $G$ and $\epsilon,\delta\in\{+1,-1\}$ tell us the orientation with which $\sigma(s_i)$ and $\sigma(t_i)$ cross their images. There are $l/2$ couples that are sent to $4L^2$ possible quadruples; but by hypothesis we have $l/2>4L^2\cdot 2d$, so we can find at least $2d+1$ couples $(s_{i_1},t_{i_1}),...,(s_{i_{2d+1}},t_{i_{2d+1}})$ with $i_1<...<i_{2d+1}$ which are sent to the same quadruple; this means that $\sigma$ sends $s_{i_1},...,s_{i_{2d+1}}$ all to the same edge of $G$ crossed with the same orientation, and $t_{i_1},...,t_{i_{2d+1}}$ all to the same edge of $G$ crossed with the same orientation.

Each of the couples $(s_{i_k},t_{i_k})$ has to contain an innermost cancellation, i.e. there is a cancellation $(q_k,r_k)$ with $\opiu{s_{i_k}}\le\opiu{q_k}=\omeno{r_k}\le\omeno{t_{i_k}}$. By Lemma \ref{innermostbounded} there are at most $2d$ innermost cancellations: since we have $2d+1$ couples $(s_{i_1},t_{i_1}),...,(s_{i_{2d+1}},t_{i_{2d+1}})$, two of them, let's say $(s_{i_j},t_{i_j})$ and $(s_{i_k},t_{i_k})$ with $j<k$, have to contain the same innermost cancellation $(q_j,r_j)=(q_k,r_k)$. But this forces $s_{i_j},s_{i_k},t_{i_k},t_{i_j}$ to appear in this order on the interval $I_l$. Thus the two couples $(s_{i_j},t_{i_j}),(s_{i_k},t_{i_k})$ are parallel, as desired.
\end{proof}
%
%\begin{myprop}\label{shorteningequation}
%Let $w\in\fI_g$ be a non-trivial equation of degree $d$ and let $\sigma:I_l\rar G$ be the corresponding path. If $l>16L^2d$, then there is a non-trivial equation $w'\in\fI_g$ of degree $d'\le d$ such that the corresponding reduced path $\sigma':I_{l'}\rar G$ has length $l'<l$. Moreover, the path $\sigma'$ can be obtained from $\sigma$ by means of a cancellation move as described in Lemma \ref{parallelcancellation}.
%\end{myprop}
%\begin{proof}
%We take any maximal reduction process $(s_1,t_1),...,(s_{l/2},t_{l/2})$ for $f\circ\sigma$ and we apply Proposition \ref{findingparallels} in order to find two parallel couples $(s_\alpha,t_\alpha),(s_\beta,t_\beta)$. We apply Lemma \ref{parallelcancellation} and we collapse some of the edges of $I_l$ in order to get an interval $I_{l'}$ with $l'<l$ and a reduced path $\sigma':I_{l'}\rar G$ with $f\circ\sigma'$ homotopically trivial. Since $\sigma'$ is a reduced path, it corresponds to a non-trivial word $w'\in H*\gen{x}$; since $f\circ\sigma'$ is homotopically trivial, we must actually have $w'\in\fI_g$. Fix an edge $e$ of $G$ with $e\in\core{\gen{g}}$; by Lemma \ref{degree}, the degree $d'$ of $w'$ is the number of occurrences of $e$ in $\sigma'$, and the degree $d$ of $w$ is the number of occurrences of $e$ in $\sigma$; but since $I_{l'}$ is obtained from $I_l$ by collapsing edges, and $\sigma'$ coincides with $\sigma$ on the other edges, the number of occurrences of $e$ on $\sigma$ is bigger or equal than the number of occurrences of $e$ on $\sigma'$. Thus $d'\le d$ as desired.
%\end{proof}

We are now ready to prove Theorem \ref{main}.

\begin{proof}[Proof of Theorem \ref{main}]
Let $w\in\fI_g$ be a non-trivial equation of degree $d_{min}$, and let $\sigma:I_l\rar G$ be the corresponding path. Suppose also that, between the equations of degree $d_{min}$, the equation $w$ has the property that the length $l$ of the corresponding path is the minimum possible. This in particular implies that $w$ is cyclically reduced.

Assume by contradiction that $l>16L^2d_{min}$. Then take any maximal reduction process $(s_1,t_1),...,(s_{l/2},t_{l/2})$ for $f\circ\sigma$, and by Proposition \ref{findingparallels} we can find two parallel couples $(s_i,t_i),(s_j,t_j)$. We perform the corresponding parallel cancellation move (according to Lemma \ref{parallelcancellation}) and we obtain a path $\sigma':I_{l'}\rar G$, with corresponding equation $w'\in\fI_g$ with $w'\not=1$. Notice that $l'<l$; moreover, by Lemma \ref{degreedecreases}, the degree $d'$ of $w'$ satisfies $d'\le d_{min}$, but since $d_{min}$ is the minimum possible this implies $d'=d_{min}$. But then $l'<l$ contradicts the minimality of $l$. This proves the theorem.
\end{proof}

\begin{mycor}\label{algorithm}
There is an algorithm that, given $H$ and $g$ such that $g$ depends on $H$, produces a non-trivial equation $w\in\fI_g$ of minimum possible degree.
\end{mycor}

\begin{proof}[Algorithm]
We first produce an upper bound $D$ on the minimum degree of an equation in $\fI_g$; this is done for example by taking any non-trivial equation in $\fI_g$, and taking its degree $D$. Given this upper bound $D$, we take all the non-trivial reduced paths $\sigma:I_l\rar G$ from the basepoint to itself and of length $l\le16L^2D$. For each such path $\sigma$, we check whether $f\circ\sigma$ is homotopically trivial (in linear time on a pushdown automaton, with a free reduction process), and we compute the degree of the corresponding equation $w$. We take the minimum of all the degrees of those equations: this is also the minimum possible degree for a non-trivial equation in $\fI_g$.
\end{proof}

\section{The set of minimum-degree equations}

In this section we describe a parallel insertion move and we show that it is an inverse to the degree-preserving cancellation moves. We also provide a few lemmas that help us manipulate sequences of insertion moves. The aim of this section is to provide an explicit characterization of the set of all the equations of minimum possible degree (and more generally, of the set of all the equations of a certain fixed degree).

\subsection{Parallel insertion}

Observe that, for every vertex $v\in\core{H}$, the group $\pi_1(\core{H},v)$ can be seen as a subgroup of $F_n$, by means of the injective map $\pi_1(f):\pi_1(\core{H},v)\rar\pi_1(R_n,*)$, where $f:\core{H}\rar R_n$ is the labeling map.

\begin{mydef}\label{insertionpath}
Let $\sigma:I_l\rar G$ be a reduced path with $\sigma(0)=\sigma(1)=*$ and such that $f\circ\sigma$ is homotopically trivial; let $(s_1,t_1),...,(s_{l/2},t_{l/2})$ be a maximal reduction process for $f\circ\sigma$. Fix a couple $(s_\alpha,t_\alpha)$ such that $\sigma(s_\alpha)$ and $\sigma(t_\alpha)$ belong to $\core{H}$; an element $u\in F_n$ is called \textbf{$\inserzione$ for $\sigma$ at $(s_\alpha,t_\alpha)$} if it satisfies the following conditions:

(i) $u$ belongs to the subgroup $\pi_1(\core{H},\sigma(\omeno{s_\alpha}))\cap\pi_1(\core{H},\sigma(\opiu{t_\alpha}))$ of $F_n$.

(ii) $u$ begins with the label of $\sigma(s_\alpha)$, if $\sigma$ crosses $\sigma(s_\alpha)$ with the same orientation of the labeling, or with the inverse of that label, if $\sigma$ crosses $\sigma(s_\alpha)$ with opposite orientation to the labeling.

(iii) $u$ is cyclically reduced.
\end{mydef}

Let $u$ be a $\inserzione$ at $(s_\alpha,t_\alpha)$ for the path $\sigma$. Then there is a unique reduced path $\tau_1:I_r\rar G$ representing $u\in\pi_1(\core{H},\sigma(\omeno{s_\alpha}))$; similarly, there is a unique reduced path $\tau_2:I_r\rar G$ representing $\ol{u}\in\pi_1(\core{H},\sigma(\opiu{t_\alpha}))$. These two paths have the same length $r$, which is also the length of the word $u$. Now cut $I_l$ at the two points $\omeno{s_\alpha}$ and $\opiu{t_\alpha}$, and insert an interval of length $r$ at each of these two cuts, in order to obtain an interval $I_{l+2r}$; define the map $\sigma':I_{l+2r}\rar G$ which is equal to $\sigma$ on the edges that belonged to $I_l$, and is equal to $\tau_1$ on the interval added at the cut at $\omeno{s_\alpha}$, and is equal to $\tau_2$ on the interval added at the cut at $\opiu{t_\alpha}$; see also figure \ref{insertionmove}.

\begin{mylemma}[Parallel insertion]\label{parallelinsertion}
Let $\sigma:I_l\rar G$ be a reduced path with $\sigma(0)=\sigma(1)=*$ and such that $f\circ\sigma$ is homotopically trivial; let $(s_1,t_1),...,(s_{l/2},t_{l/2})$ be a maximal reduction process for $f\circ\sigma$. Let $(s_\alpha,t_\alpha)$ be a couple such that $\sigma(s_\alpha)$ and $\sigma(t_\alpha)$ belong to $\core{H}$ and let $u$ be an $\inserzione$ for $\sigma$ at $(s_\alpha,t_\alpha)$. Let $\sigma':I_{l+2r}\rar G$ be the path defined as above. Then $\sigma'$ is a reduced path with $\sigma'(0)=\sigma'(1)=*$ and $f\circ\sigma'$ is homotopically trivial. Moreover, there is a maximal reduction process for $f\circ\sigma'$ containing the couples $(s_1,t_1),...,(s_{l/2},t_{l/2})$.
\end{mylemma}
\begin{proof}
Of course we have $\sigma'(0)=\sigma'(1)=*$. The fact that $\sigma'$ is reduced follows from the fact that $\sigma$ is reduced, and from the fact that $u$ is cyclically reduced (here it is important that $u$ is cyclically reduced and begins with the label of $\sigma(s_\alpha)$ or with its inverse: otherwise the local injectivity of $\sigma'$ may fail at $\omeno{s_\alpha}$). Let $e_r,...,e_1$ be the edges of the interval which is the domain of $\tau_1$ and let $e_1',...,e_r'$ be the edges of the interval which is the domain of $\tau_2$; we mean that $e_r,...,e_1$ and $e_1',...,e_r'$ appear in this order on $I_{l+2r}$. Then a maximal reduction process for $f\circ\sigma'$ is given by $(s_1,t_1),...,(s_\alpha,t_\alpha),(e_1,e_1'),...,(e_r,e_r'),(s_{\alpha+1},t_{\alpha+1 }),...,(s_{l/2},t_{l/2})$, and in particular $f\circ\sigma'$ is homotopically trivial, as desired.
\end{proof}

\begin{myrmk}
Notice that these moves of parallel insertion depend on the existence of an element $u\in\pi_1(\core{H},\sigma(\omeno{s_\alpha}))\cap\pi_1(\core{H},\sigma(\opiu{t_\alpha}))$ with some specific properties. The two subgroups $\pi_1(\core{H},\sigma(\omeno{s_\alpha}))$ and $\pi_1(\core{H},\sigma(\opiu{t_\alpha}))$ are both conjugates of $H$, so there are cases where the possibilities for $u$ are very limited (for example if $H$ is malnormal in $F_n$, meaning that every two distinct conjugates of $H$ have trivial intersection). In any case it is possible that $\sigma(\omeno{s_\alpha})=\sigma(\opiu{t_\alpha})$, giving the possibility for at least some insertion moves to be performed.
\end{myrmk}

\begin{figure}[h!]
\centering
\begin{tikzpicture}[scale=1]
\node (0) at (0,0) {.};
\node (1) at (1,0) {.};
\node (2) at (2,0) {.};
\node (3) at (3,0) {.};
\node (4) at (4,0) {.};
\node (5) at (5,0) {.};
\node (6) at (6,0) {.};
\node (7) at (7,0) {.};
\node (8) at (8,0) {.};
\draw (0) to (1);
\draw (1) to (2);
\draw (2) to (3);
\draw (3) to (4);
\draw (4) to (5);
\draw (5) to (6);
\draw (6) to (7);
\draw (7) to (8);

\draw[out=90,in=90,looseness=0.7,densely dotted] (0.5,0) to (7.5,0);
\draw[out=90,in=90,looseness=1,densely dotted] (1.5,0) to (4.5,0);
\draw[out=90,in=90,looseness=1,densely dotted] (2.5,0) to (3.5,0);
\draw[out=90,in=90,looseness=1,densely dotted] (5.5,0) to (6.5,0);

\node (sa) at (1.5,-0.3) {$s_\alpha$};
\node (ta) at (4.5,-0.3) {$t_\alpha$};

\begin{scope}[yshift=-3cm]
\node (0i) at (-0.3,0) {.};
\node (1i') at (0.7,0) {.};
\node (1i) at (1,0) {.};
\node (2i) at (2,0) {.};
\node (3i) at (3,0) {.};
\node (4i) at (4,0) {.};
\node (5i) at (5,0) {.};
\node (5i') at (5.3,0) {.};
\node (6i) at (6.3,0) {.};
\node (7i) at (7.3,0) {.};
\node (8i) at (8.3,0) {.};
\draw (0i) to (1i');
\draw (1i) to (2i);
\draw (2i) to (3i);
\draw (3i) to (4i);
\draw (4i) to (5i);
\draw (5i') to (6i);
\draw (6i) to (7i);
\draw (7i) to (8i);

\draw[out=90,in=90,looseness=0.7,densely dotted] (0.2,0) to (7.8,0);
\draw[out=90,in=90,looseness=1,densely dotted] (1.5,0) to (4.5,0);
\draw[out=90,in=90,looseness=1,densely dotted] (2.5,0) to (3.5,0);
\draw[out=90,in=90,looseness=1,densely dotted] (5.8,0) to (6.8,0);

\node (sai) at (1.5,-0.3) {$s_\alpha$};
\node (tai) at (4.5,-0.3) {$t_\alpha$};

\node (9i) at (-1.8,-1) {.};
\node (10i) at (-0.8,-1) {.};
\node (11i) at (0.2,-1) {.};
\node (12i) at (1.2,-1) {.};
\node (13i) at (4.8,-1) {.};
\node (14i) at (5.8,-1) {.};
\node (15i) at (6.8,-1) {.};
\node (16i) at (7.8,-1) {.};

\draw (9i) to (10i);
\draw (10i) to (11i);
\draw (11i) to (12i);
\draw (13i) to (14i);
\draw (14i) to (15i);
\draw (15i) to (16i);
\draw[densely dashed] (9i) to (1i');
\draw[densely dashed] (12i) to (1i);
\draw[densely dashed] (13i) to (5i);
\draw[densely dashed] (16i) to (5i');
\draw[->] (-1.8,-1.2) to node[below]{$u$} (1.2,-1.2);
\draw[->] (7.8,-1.2) to node[below]{$u$} (4.8,-1.2);
\end{scope}

\begin{scope}[yshift=-7.7cm,xshift=-3cm]
\node (0j) at (0,0) {.};
\node (1j) at (1,0) {.};
\node (2j) at (2,0) {.};
\node (3j) at (3,0) {.};
\node (4j) at (4,0) {.};
\node (5j) at (5,0) {.};
\node (6j) at (6,0) {.};
\node (7j) at (7,0) {.};
\node (8j) at (8,0) {.};
\node (9j) at (9,0) {.};
\node (10j) at (10,0) {.};
\node (11j) at (11,0) {.};
\node (12j) at (12,0) {.};
\node (13j) at (13,0) {.};
\node (14j) at (14,0) {.};
\draw (0j) to (1j);
\draw (1j) to (2j);
\draw (2j) to (3j);
\draw (3j) to (4j);
\draw (4j) to (5j);
\draw (5j) to (6j);
\draw (6j) to (7j);
\draw (7j) to (8j);
\draw (8j) to (9j);
\draw (9j) to (10j);
\draw (10j) to (11j);
\draw (11j) to (12j);
\draw (12j) to (13j);
\draw (13j) to (14j);

\draw[out=90,in=90,looseness=0.6,densely dotted] (0.5,0) to (13.5,0);
\draw[out=90,in=90,looseness=0.9,densely dotted] (4.5,0) to (7.5,0);
\draw[out=90,in=90,looseness=1,densely dotted] (5.5,0) to (6.5,0);
\draw[out=90,in=90,looseness=1,densely dotted] (11.5,0) to (12.5,0);

\draw[out=90,in=90,looseness=0.7,densely dotted] (1.5,0) to (10.5,0);
\draw[out=90,in=90,looseness=0.75,densely dotted] (2.5,0) to (9.5,0);
\draw[out=90,in=90,looseness=0.8,densely dotted] (3.5,0) to (8.5,0);

\node (sai) at (4.5,-0.3) {$s_\alpha$};
\node (tai) at (7.5,-0.3) {$t_\alpha$};
\draw[->] (1,-0.2) to node[below]{$u$} (4,-0.2);
\draw[->] (11,-0.2) to node[below]{$u$} (8,-0.2);
\end{scope}

\end{tikzpicture}
\caption{An example of an insertion move. In the image above, we can see a diagram for a maximal reduction process for $\sigma$. In the image in the middle, we see the two cuts at $\omeno{s_\alpha}$ and $\opiu{t_\alpha}$ and the two pieces $u$ and $\ol{u}$ ready to be inserted. In the image below, we see the result after the insertion move, and a diagram for a maximal reduction process for $\sigma'$.}
\label{insertionmove}
\end{figure}
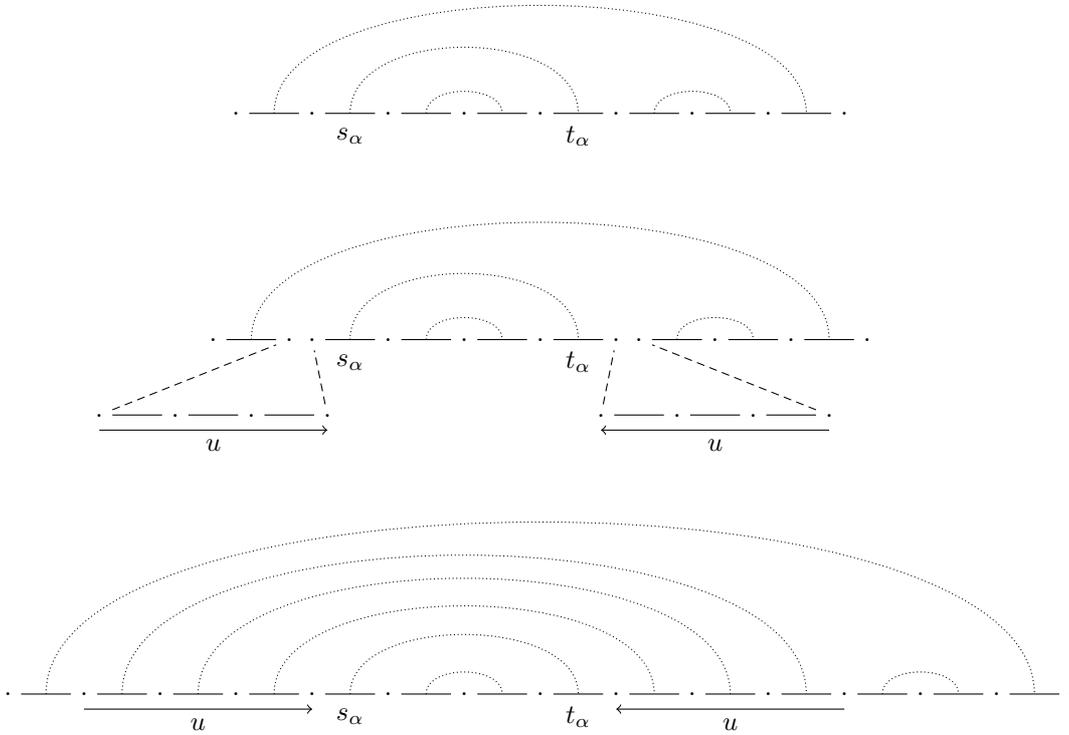

The following two lemmas show that the parallel insertion moves of Lemma \ref{parallelinsertion} are essentially the inverse of the parallel cancellation moves of Lemma \ref{parallelcancellation} which are degree-preserving as in Definition \ref{defdegpres}.

\begin{mylemma}\label{parallelinverses}
Let $\sigma:I_l\rar G$ be a cyclically reduced path with $\sigma(0)=\sigma(1)=*$ and such that $f\circ\sigma$ is homotopically trivial; let $(s_1,t_1),...,(s_{l/2},t_{l/2})$ be a maximal reduction process for $f\circ\sigma$. Suppose there are two parallel couples $(s_\alpha,t_\alpha),(s_\beta,t_\beta)$ and let $\sigma':I_{l'}\rar G$ be the path obtained with the cancellation move of Lemma \ref{parallelcancellation}. Suppose that the cancellation move is degree-preserving, and let $u$ be the word that we read when going along $\sigma([\omeno{s_\beta},\omeno{s_\alpha}])$. Then $\sigma$ can be obtained from $\sigma'$ with an insertion move as described in Lemma \ref{parallelinsertion}, using the $\inserzione$ $u$ for $\sigma'$ at $(s_\alpha,t_\alpha)$.
\end{mylemma}
\begin{proof}
Immediate from Lemma \ref{degreepreserving}.
\end{proof}

\begin{mylemma}\label{parallelinverses2}
Let $\sigma':I_l\rar G$ be a cyclically reduced path with $\sigma'(0)=\sigma'(1)=*$ and such that $f\circ\sigma'$ is homotopically trivial; let $(s_1,t_1),...,(s_{l/2},t_{l/2})$ be a maximal reduction process for $f\circ\sigma'$. Let $(s_\alpha,t_\alpha)$ be a couple such that $\sigma'(s_\alpha)$ and $\sigma'(t_\alpha)$ belong to $\core{H}$ and let $u$ be an $\inserzione$ for $\sigma'$ at $(s_\alpha,t_\alpha)$; let $\sigma$ be the path obtained from $\sigma'$ by means of the insertion move of Lemma \ref{parallelinsertion}. Then $\sigma'$ can be obtained from $\sigma$ by means of a cancellation move which collapses the intervals that we just added; moreover this cancellation move is degree-preserving.
\end{mylemma}
\begin{proof}
Immediate from the definitions.
\end{proof}

We are now going to prove the technical Lemmas \ref{insertioncommutes}, \ref{insertionsame} and \ref{insertioninsertion}; these will allow us to manipulate a sequence of insertion moves. The following lemma says that, if we take a path $\sigma$ and we have two parallel insertion moves that we want to perform on $\sigma$, then we can perform them in any order that we want, and we get the same result.

\begin{mylemma}\label{insertioncommutes}
Let $\sigma:I_l\rar G$ be a reduced path with $\sigma(0)=\sigma(1)=*$ and such that $f\circ\sigma$ is homotopically trivial; let $(s_1,t_1),...,(s_{l/2},t_{l/2})$ be a maximal reduction process for $f\circ\sigma$. Let $(s_\alpha,t_\alpha),(s_{\alpha'},t_{\alpha'})$ be distinct couples such that $\sigma(s_\alpha),\sigma(t_\alpha),\sigma(s_{\alpha'}),\sigma(t_{\alpha'})$ belong to $\core{H}$, and let $u,u'$ be $\inserzione$s for $\sigma$ at $(s_\alpha,t_\alpha),(s_{\alpha'},t_{\alpha'})$ respectively. Perform on $\sigma$ the insertion move relative to $u,(s_\alpha,t_\alpha)$ and then the insertion move relative to $u',(s_{\alpha'},t_{\alpha'})$ in order to obtain a path $\mu_1$. Perform on $\sigma$ the insertion move relative to $u',(s_{\alpha'},t_{\alpha'})$ and then the insertion move relative to $u,(s_\alpha,t_\alpha)$ in order to get a path $\mu_2$. Then $\mu_1$ and $\mu_2$ are the same path.
\end{mylemma}
\begin{proof}
The two domains of $\mu_1,\mu_2$ are defined starting with the same interval $I_l$, and adding edges as explained in Lemma \ref{parallelinsertion}. The edges added are the same, and the maps $\mu_1,\mu_2$ are defined in the same way on those edges. The only thing that changes is the order in which the edges are added, but the resulting paths $\mu_1$ and $\mu_2$ are the same.
\end{proof}

The following lemma says that, if we take a path and we perform two parallel insertion moves at the same couple of edges, then we can consolidate then into one single insertion move instead (at the same couple of edges).

\begin{mylemma}\label{insertionsame}
Let $\sigma:I_l\rar G$ be a reduced path with $\sigma(0)=\sigma(1)=*$ and such that $f\circ\sigma$ is homotopically trivial; let $(s_1,t_1),...,(s_{l/2},t_{l/2})$ be a maximal reduction process for $f\circ\sigma$. Let $(s_\alpha,t_\alpha)$ be a couple in the reduction process such that $\sigma(s_\alpha),\sigma(t_\alpha)$ belong to $\core{H}$, and let $u,u'$ be $\inserzione$s for $\sigma$ at $(s_\alpha,t_\alpha)$. Perform on $\sigma$ the insertion move relative to $u,(s_\alpha,t_\alpha)$ and then the insertion move relative to $u',(s_\alpha,t_\alpha)$ in order to obtain a path $\mu_1$. Perform on $\sigma$ the insertion move relative to $uu',(s_\alpha,t_\alpha)$ in order to obtain a path $\mu_2$. Then $\mu_1$ and $\mu_2$ are the same path.
\end{mylemma}
\begin{proof}
Completely analogous to the proof of Lemma \ref{insertioncommutes}.
\end{proof}

The following Lemma \ref{insertioninsertion} says that, if we take a path and we perform a parallel insertion move at a couple of edges, and then another insertion move at a couple of edges that we just added, then we can again consolidate the two insertion moves into a single one. Notice that this is slightly different from the previous Lemma \ref{insertionsame}.

\begin{mylemma}\label{insertioninsertion}
Let $\sigma:I_l\rar G$ be a reduced path with $\sigma(0)=\sigma(1)=*$ and such that $f\circ\sigma$ is homotopically trivial; let $(s_1,t_1),...,(s_{l/2},t_{l/2})$ be a maximal reduction process for $f\circ\sigma$. Let $(s_\alpha,t_\alpha)$ be a couple such that $\sigma(s_\alpha)$ and $\sigma(t_\alpha)$ belong to $\core{H}$, and let $u$ be an $\inserzione$ for $\sigma$ at $(s_\alpha,t_\alpha)$; let $\sigma'$ be the reduced path obtained with the insertion move relative to $u,(s_\alpha,t_\alpha)$, and take a maximal reduction process for $f\circ\sigma'$ containing all the couples $(s_1,t_1),...,(s_{l/2},t_{l/2})$. Let $(s',t')$ be a couple in the reduction process for $f\circ\sigma'$ that does not belong to the reduction process for $f\circ\sigma$, and let $u'$ be an $\inserzione$ for $\sigma'$ at $(s',t')$; let $\sigma''$ be the path obtained from $\sigma'$ after performing the insertion move relative to $u',(s',t')$. Then there is an $\inserzione$ $\ol u$ for $\sigma$ at $(s_\alpha,t_\alpha)$ such that, if we perform on $\sigma$ the insertion move relative to $\ol{u},(s_\alpha,t_\alpha)$, we obtain $\sigma''$.
\end{mylemma}
\begin{proof}
Completely analogous to the proof of Lemma \ref{insertioncommutes}.
\end{proof}

\subsection{Characterization of all the minimum-degree equations}

Recall that $L$ is the number of edges of $G$. Let $d_{min}$ be the minimum possible degree for a non-trivial equation $w\in\fI_g$.

\begin{mythm}\label{main2}
Let $w\in\fI_g$ be a cyclically reduced equation of degree $d_{min}$ and let $\sigma:I_l\rar G$ be the corresponding reduced path. Then there is a cyclically reduced equation $w'\in\fI_g$ of degree $d_{min}$ with corresponding path $\sigma':I_{l'}\rar G$, and a maximal reduction process for $\sigma'$, such that:

(i) The path $\sigma'$ has length $l'\le 16L^2d_{min}$.

(ii) The path $\sigma$ can be obtained from $\sigma'$ by means of at most $l'/2$ insertion moves (as in Lemma \ref{parallelinsertion}), each of them performed on a distinct couple of edges of $I_{l'}$.
\end{mythm}
\begin{proof}
If the length of $\sigma$ is $l>16L^2d_{min}$, then by Proposition \ref{findingparallels} we can perform a cancellation move on $\sigma$ in order to get a shorter path. The degree can't strictly increase, by Lemma \ref{degreedecreases}, and can't strictly decrease, since $d_{min}$ was minimum. Thus we obtain a strictly shorter path, whose corresponding equation has the same degree $d_{min}$. We reiterate the process, and after a finite number of parallel cancellation moves we have to obtain a path $\sigma':I_{l'}\rar G$ with corresponding equation of degree $d_{min}$ and of length $l'\le16L^2d_{min}$.

Since $\sigma'$ is obtained from $\sigma$ by means of a finite number of parallel cancellation moves, by Lemma \ref{parallelinverses} this means that $\sigma$ can be obtained from $\sigma'$ by means of a sequence of insertion moves of Lemma \ref{parallelinsertion}. Take a sequence of insertion moves $\iota_1,...,\iota_p$ that changes $\sigma'$ into $\sigma$, and has minimum length $p$ between all such sequences.

Suppose there are two insertions $\iota_q,\iota_r$ with $q<r$ such that $\iota_r$ acts on a couple of edges that is added by $\iota_q$: then we take an innermost couple of insertions with that property, so that each transformation $\iota_j$ with $q<j<r$ acts on a couple of edges different from $\iota_r$. In particular, by Lemma \ref{insertioncommutes}, we can change the order in our sequence in order to bring $\iota_r$ adjacent to $\iota_q$, and we can then apply Lemma \ref{insertioninsertion} in order to substitute $\iota_q,\iota_r$ with a single insertion move. This contradicts the minimality of the length $p$ of the sequence.

Thus in our sequence $\iota_1,...,\iota_q$ we have that each insertion move acts on a couple of edges of the original interval of definition $I_{l'}$ of $\sigma'$. If two insertion moves $\iota_q,\iota_r$ with $q<r$ act on the same couple of edges of $I_{l'}$, then we reason as above, and by means of Lemmas \ref{insertioncommutes} and \ref{insertionsame} we can substitute them with a single insertion move, contradicting the minimality of $p$.

It follows that each couple of insertion moves of the sequence $\iota_1,...,\iota_p$ acts on a different couple of edges of the original interval of definition $I_{l'}$ of $\sigma'$, and in particular $p\le l'/2$. The conclusion follows.
\end{proof}

\subsection{Equations of an arbitrary fixed degree}

Until now we focused on the study of the equations of minimum possible degree, but the results can be generalized to equations of any fixed degree. Let $d\ge1$ be an integer. Let $L$ be the number of edges of the graph $G$. The following proposition is similar to Proposition \ref{findingparallels}, but with the difference that this time we are looking for a parallel cancellation move which is degree-preserving.

\begin{myprop}\label{findingparallels2}
Let $w\in\fI_g$ be a cyclically reduced equation of degree $d$ and let $\sigma:I_l\rar G$ be the corresponding path; let $(s_1,t_1),...,(s_{l/2},t_{l/2})$ be any maximal reduction process for $f\circ\sigma$. Suppose $l>32L^4d^2+16L^3d$. Then the reduction process contains two $\parallele$ couples $(s_\alpha,t_\alpha),(s_\beta,t_\beta)$ such that the corresponding parallel cancellation move is degree-preserving.
\end{myprop}
\begin{proof}
Take the domain $I_l$ of $\sigma$ and remove all the edges $r$ with the following property: $r$ belongs to a couple $(s_i,t_i)$ such that at least one of $\sigma(s_i),\sigma(t_i)$ is an edge of $\bcore{\gen{g}}$. By Lemma \ref{degree}, for every edge $e$ of $\bcore{\gen{g}}$ there are exactly $d$ edges of $I_l$ that are sent to $e$; and $\bcore{\gen{g}}$ contains at most $L$ edges. This means that we removed from $I_l$ at most $2Ld$ edges, and thus there remain at most $2Ld+1$ connected components, which we call $C_1,...,C_a$ for $a\le2Ld+1$. Since $l\ge16L^3d(2Ld+1)+1$, there is at least one connected component $\ol{C}\in\{C_1,...,C_a\}$ of length at least $16L^3d+1$.

We observe that there is no couple $(s_i,t_i)$ with both $s_i,t_i\in\ol{C}$: otherwise, the interval $[\omeno{s_i},\opiu{t_i}]\subseteq\ol{C}$ would contain an innermost cancellation, and thus by Lemma \ref{innermostcouple} we would find an edge of $\ol{C}$ which is sent to $\bcore{\gen{g}}$, contradiction.

Suppose the connected component $\ol{C}$ contains at least $8L^3d+1$ edges $s_i$ belonging to couples $(s_i,t_i)$ of the cancellation process (otherwise $\ol{C}$ has to contain at least $8L^3d+1$ edges $t_i$ belonging to couples $(s_i,t_i)$ of the reduction process, and the reasoning is analogous). To each such edge $s_i$, we associate the quintuple $(\sigma(s_i),\epsilon,\sigma(t_i),\delta,C_k)$ where $\sigma(s_i),\sigma(t_i)$ are edges of $G$ and $\epsilon,\delta\in\{+1,-1\}$ tell us the orientation with which $\sigma(s_i)$ and $\sigma(t_i)$ cross their images, and $C_k\in\{C_1,...,C_a\}\setminus\{\ol{C}\}$ is the connected component which $t_i$ belongs to. Since we have at least $8L^3d+1$ edges $s_i$ in $\ol{C}$ and at most $L\cdot 2\cdot L\cdot 2\cdot (2Ld)$ possible quintuples, there are at least two edges $s_{i_1},s_{i_2}$ with the same associated quintuple $(\sigma(s_{i_1}),\epsilon,\sigma(t_{i_1}),\delta,C_k)=(\sigma(s_{i_2}),\epsilon,\sigma(t_{i_2}),\delta,C_k)$.

It immediately follows that $(s_{i_1},t_{i_1})$ and $(s_{i_2},t_{i_2})$ are parallel couples. Without loss of generality we can assume that $\omeno{s_{i_1}}<\omeno{s_{i_2}}$; we have that the interval $[\omeno{s_{i_1}},\omeno{s_{i_2}}]$ is contained in $\ol{C}$ and the interval $[\opiu{t_{i_2}},\opiu{t_{i_1}}]$ is contained in $C_k$. Thus, when we perform the cancellation move relative to the parallel couples $(s_{i_1},t_{i_1})$ and $(s_{i_2},t_{i_2})$, we only remove edges whose image is in $\bcore{H}$. We conclude from Lemma \ref{degree} that the cancellation move is degree-preserving, as desired.
\end{proof}

We are now ready to state and prove the analogues to Theorem \ref{main} and to Theorem \ref{main2}.

\begin{mythm}\label{main3}
Suppose $\fI_g$ contains a non-trivial equation of degree $d$. Then $\fI_g$ contains a non-trivial equation $w$ of degree $d$ such that the corresponding path $\sigma:I_l\rar G$ has length $l\le32L^4d^2+16L^3d$.
\end{mythm}
\begin{proof}
Take any non-trivial equation in $\fI_g$ of degree $d$ and such that the corresponding path $\sigma:I_l\rar G$ has minimum length $l$; in particular this implies that the equation is cyclically reduced. If $l>32L^4d^2+16L^3d$ then by Proposition \ref{findingparallels2} we can perform a degree-preserving cancellation move on $\sigma$, and thus we can find a non-trivial equation in $\fI_g$ of degree $d$ whose corresponding path is strictly shorter, contradiction. Thus we must have $l\le32L^4d^2+16L^3d$, and the conclusion follows.
\end{proof}

\begin{mycor}\label{algorithm3}
There is an algorithm that, given $H\sgr F_n$ finitely generated and $g\in F_n$ and an integer $d\ge1$, tells us whether $\fI_g$ contains non-trivial equations of degree $d$, and, if so, produces an equation $w\in\fI_g$ of degree $d$.
\end{mycor}

\begin{mythm}\label{main4}
Let $w\in\fI_g$ be a cyclically reduced equation of degree $d$ and let $\sigma:I_l\rar G$ be the corresponding reduced path. Then there is a cyclically reduced equation $w'\in\fI_g$ of degree $d$ with corresponding path $\sigma':I_{l'}\rar G$, and a maximal reduction process for $\sigma'$, such that:

(i) The path $\sigma'$ has length $l'\le32L^4d^2+16L^3d$.

(ii) The path $\sigma$ can be obtained from $\sigma'$ by means of at most $l'/2$ insertion moves (as in Lemma \ref{parallelinsertion}), each of them performed on a distinct couple of edges of $I_{l'}$.
\end{mythm}
\begin{proof}
Completely analogous to the proof of Theorem \ref{main2}.
\end{proof}

\subsection{The set of possible degrees}\label{SubsectionD}

Let $H\sgr F_n$ be a finitely generated subgroup, let $g\in F_n$ be an element that depends on $H$, and let $\fI_g\nor H*\gen{x}$ be the ideal of the equations for $g$ over $H$.

\begin{mydef}\label{defDg}
Define $D_g=\{d\in\bN :\text{ there is a non-trivial equation }w\in\fI_g\text{ of degree }d\}$.
\end{mydef}

\begin{mylemma}\label{obtainingdegrees}
If $d,d'\in D_g$ and $k\ge0$ then $d+d'+2k\in D_g$.
\end{mylemma}
\begin{proof}
Let $w\in\fI_g$ be an equation of degree $d$: up to cyclic permutation, we can assume that $w$ is of the form $c_1x^{e_1}...c_\alpha x^{e_\alpha}$ with $c_1,...,c_\alpha\in H\setminus\{1\}$ and $e_1,...,e_\alpha\in\bZ\setminus\{0\}$. Similarly, let $w'\in\fI_g$ be an equation of degree $d'$, and similarly we assume that $w'=c_1'x^{e_1'}...c_{\beta}'x^{e_\beta'}$ with $c_1',...,c_\beta'\in H\setminus\{1\}$ and $e_1',...,e_\beta'\in\bZ\setminus\{0\}$. Without loss of generality, also assume that $e_\beta'>0$ and we take $h\in H\setminus\{1,c_1\}$. Then $w''=\ol{h}wh\ol{x}^kw'x^k$ belongs to $\fI_g$ and has degree $d+d'+2k$, for any $k\ge0$. The conclusion follows.
\end{proof}

Denote with $2\bN$ the set of non-negative even numbers.

\begin{mythm}\label{degreeset}
Exactly one of the following possibilities takes place:

(i) $D_g$ contains an odd number and $\bN\setminus D_g$ is finite.

(ii) $D_g$ contains only even numbers and $2\bN\setminus D_g$ is finite.
\end{mythm}
\begin{proof}
If $\fI_g$ contains only equations of even degree, then we take any equation of even degree $d$, and by Lemma \ref{obtainingdegrees} we are able to obtain equations of degree $d+d+2k$ for every $k\ge0$. Thus in this case we have that $2\bN\setminus D_g$ is finite.

Suppose now $\fI_g$ contains an equation of odd degree $d$. Then by Lemma \ref{obtainingdegrees} we are able to obtain equations of degree $d+d+2k$ for every $k\ge0$, and thus equations of every even degree big enough. In particular we are able to obtain an equation of degree $2d$, and thus by Lemma \ref{obtainingdegrees} we are able to obtain equations of degree $2d+d+2k$ for every $k\ge0$, and thus equations of every odd degree big enough. Thus in this case we have that $\bN\setminus D_g$ is finite.
\end{proof}

In order to understand whether we fall into case (i) or (ii) of Theorem \ref{degreeset}, it is enough to look at a set of normal generators for $\fI_g$.

\begin{mylemma}\label{evendegree}
Let $H\sgr F_n$ be a finitely generated subgroup and let $g\in F_n$ be an element that depends on $H$. Suppose the set of equations $W\subseteq\fI_g$ generates $\fI_g$ as normal subgroup of $H*\gen{x}$, and suppose every equation $w\in W$ has even degree. Then every equation in $\fI_g$ has even degree.
\end{mylemma}
\begin{proof}
Consider the homomorphism $\phi:H*\gen{x}\rar\bZ/2\bZ$ defined by $\phi(h)=0$ for every $h\in H$ and $\phi(x)=1$. Observe that $\phi$ sends equations of even degree to $0$ and equations of odd degree to $1$. Since every $w\in W$ has even degree, we have that $W$ is contained in $\ker\phi$. But then the normal subgroup $\fI_g$ generated by $W$ is contained in $\ker\phi$ too, and thus $\fI_g$ only contains equations of even degree.
\end{proof}
%\begin{proof}
%Fix a basis $h_1,...,h_r$ for $H$ and use the basis $h_1,...,h_r,x$ for $H*\gen{x}$. Without loss of generality, we can assume that $w_1,...,w_k$ are cyclically reduced words in $h_1,...,h_r,x$. A generic equation $w\in\fI_g$ can be written as a finite product of conjugates of the generators $w_1,...,w_k$, let's say $w=v_1w_{i_1}\ol{v_1}v_2w_{i_2}\ol{v_2}...v_sw_{i_s}\ol{v_s}$ where $v_1,...,v_s$ are reduced words in the letters $h_1,...,h_r,x$. Write down the concatenation of words $v_1w_{i_1}\ol{v_1}v_2w_{i_2}\ol{v_2}...v_sw_{i_s}\ol{v_s}$ without reducing it, and count the number of occurrences of $x$ and $\ol{x}$: this is an even number (because each $w_{i_j}$ has an even number of occurrences by hypothesis, and $v_j$ and $\ol{v_j}$ have the same number of occurrences). We now apply a cyclic reduction process to the concatenation of words $v_1w_{i_1}\ol{v_1}v_2w_{i_2}\ol{v_2}...v_sw_{i_s}\ol{v_s}$: at each step, if we cancel a $h_i$ against a $\ol{h_i}$ we have no change in the number of occurrences of $x$ and $\ol{x}$; and if we cancel a $x$ against an $\ol{x}$ we have that the number of occurrences of $x$ and $\ol{x}$ decreases by exactly $2$. Thus the number of occurrences of $x$ and $\ol{x}$ remains even along the whole cyclic reduction process, and in particular the cyclic reduction of $w$ has an even number of occurrences of $x$ and $\ol{x}$, meaning that the degree of $w$ is even, as desired.
%\end{proof}

\begin{mythm}\label{algorithm4}
Given $H\sgr F_n$ finitely generated and $g\in F_n$ that depends on $H$, there is an algorithm that:

(a) Determines whether we fall into case (i) or (ii) of Theorem \ref{degreeset}.

(b) Computes the finite set $\bN\setminus D_g$ or $2\bN\setminus D_g$ respectively.
\end{mythm}
\begin{proof}
Let $\fI_g=\ggen{w_1,...,w_k}$ be a finite set of normal generators for $\fI_g$, which can be obtained with the algorithm of Theorem \ref{idealfingen}. According to Lemma \ref{evendegree}, if one of $w_1,...,w_k$ has odd degree then we fall into case (i) of Theorem \ref{degreeset}, otherwise we fall in case (ii) of Theorem \ref{degreeset}.

If we fall into case (i), then we take $w_i$ of degree $d_i$ odd, and with the same proof of Theorem \ref{degreeset} we have that $\bN\setminus D_g\subseteq\{1,...,3d_i\}$. For each degree $d\in\{1,...,3d_i\}$ we use Corollary \ref{algorithm3} to determine whether $d$ belongs to $D_g$. If we fall into case (ii), we perform an analogous procedure.
\end{proof}

\section{Examples}\label{SectionExamples}

We now provide a few examples for the reader, to illustrate the techniques introduced in the present paper. In each example $D_g$ denotes the set introduced in Definition \ref{defDg} and $d_{min}$ denotes the minimum of $D_g$.

\subsection{Cyclic subgroups}\label{examplecyclic}

Let $F_n=\gen{a_1,...,a_n}$ and suppose that $H$ has rank $1$, let's say $H=\gen{h}$ for some $h\in F_n$ with $h\not=1$. In order for an element $g$ to depend on $H$, we must have that $g,h$ belong to the same cyclic subgroup of $F_n$. We can use $\gen{H,g}$ as ambient free group instead of $F_n$: without loss of generality, in the following we assume that $F_n=\gen{a}$ and that $H=\gen{h}$ where $h=a^m$ with $m\ge1$, and that $g=a^k$ with $k\ge0$ coprime with $m$.

The graph $G=\bcore{H}\vee\bcore{\gen{g}}$ here has rank $2$ while $\bcore{\gen{H,g}}$ has rank $1$. This means that the algorithm of Theorem \ref{idealfingen} produces a single generator for the ideal $\fI_g\nor H*\gen{x}$. One possible such generator $w_{m,k}$ for each $m\ge1$ and $k\ge0$ coprime can be obtained by means of the following recursive formula:
$$\begin{cases}
w_{1,0}(h,x)=\ol{x}\\
w_{m,k}(h,x)=w_{m-k,k}(h\ol{x},x) \text{ for } m>k\\
w_{m,k}(h,x)=w_{m,k-m}(h,x\ol{h}) \text{ for } m\le k
\end{cases}$$
Moreover, with this definition it is possible to prove by induction that $w_{m,k}(h,x)$ contains $k$ occurrences of $h$, no occurrence of $\ol{h}$, no occurrence of $x$, and $m$ occurrences of $\ol{x}$. In particular $w_{m,k}\in\fI_g$ is an equation of degree $m$.

\begin{myrmk}
In the case $h=a^5$ and $x=a^2$ we have the generator $w_{5,2}(h,x)=h\ol{x}^2h\ol{x}^3$ for the ideal $\fI_g$. We observe that the most immediate candidate $h^2\ol{x}^5$ doesn't work, because it is contained in $\fI_g$ but it doesn't generate the whole ideal.
\end{myrmk}

\begin{myrmk}
The following is a well-known property of one-relator groups due to Magnus: if two elements of a free group generate the same normal subgroup, then they coincide, up to conjugation and inverse. In particular, the generator $w_{m,k}$ defined above is essentially the unique generator for the ideal $\fI_g$.
\end{myrmk}

Let $w\in\gen{h,x}$ be a non-trivial cyclically reduced element, which up to conjugation can be written in the form $w=h^{e_1}x^{f_1}...h^{e_r}x^{f_r}$ with $r\ge1$ and $e_1,...,e_r,f_1,...,f_r\in\bZ\setminus\{0\}$. The condition $w\in\fI_g$ is equivalent to $(e_1+...+e_r)m+(f_1+...+f_r)k=0$, and since $m,k$ are coprime this means that for some $p\in\bZ$ we have $f_1+...+f_r=pm$ and $e_1+...+e_r=-pk$. The degree of the equation is $d=\abs{f_1}+...+\abs{f_r}$. %In particular notice that either $f_1+...+f_r=0$, in which case the degree $d$ is even, or $d=\abs{f_1}+...+\abs{f_r}\ge\abs{f_1+...+f_r}\ge m$.

Suppose $m=1$. Then we have $w_{1,k}=h^k\ol{x}$. In this case $d_{min}=1$ and $D_g=\bN\setminus\{0\}$.

Suppose $m\ge2$ is even. Then $d_{min}=2$ and $D_g=2\bN\setminus\{0\}$. For the $\supseteq$ inclusion, we have the equation $[h,x^s]$ of degree $2s$ for each $s\ge1$. For the $\subseteq$ inclusion, notice that the unique generator $w_{m,k}$ has even degree, and thus by Lemma \ref{evendegree} each equation has even degree.

Suppose $m\ge2$ is odd. Then $d_{min}=2$ and $D_g=\{d : d\ge2$ even$\}\cup\{d : d\ge m$ odd$\}$. For the $\supseteq$ inclusion, we have the equation $[h,x]$ of degree $2$ and the equation $w_{m,k}$ of degree $m$, and we can use Lemma \ref{obtainingdegrees}. For the $\subseteq$ inclusion, we notice that, if an equation is written in the form $w=h^{e_1}x^{f_1}...h^{e_r}x^{f_r}$ as above, then either $f_1+...+f_r=0$, in which case the degree $d=\abs{f_1}+...+\abs{f_r}$ is even, or $m\le\abs{f_1+...+f_r}\le\abs{f_1}+...+\abs{f_r}=d$.

\subsection{An ideal with only even-degree equations}\label{example46}

Let $F_2=\gen{a,b}$ and consider the subgroup $H=\gen{h_1,h_2}$ with $h_1=ba$ and $h_2=ab^2\ol{a}$ and the element $g=a$. We can build the corresponding graph $G$, see figure \ref{example46G}, and we have that $\pi_1(G,*)$ is a free group with three generators $[\mu_{h_1}],[\mu_{h_2}],[\mu_g]$, which are the homotopy classes of the reduced paths $\mu_{h_1},\mu_{h_2},\mu_g$ corresponding to the elements $h_1,h_2\in H$ and $g$ respectively. We can perform a sequence of rank-preserving folding operations on $G$, see figure \ref{example46fold}, and we end up with a rose $R'$ with one $a$-labeled edge $e_1$ and two $b$-labeled edges $e_2,e_3$. Let $p:(G,*)\rar(R',*)$ be the map given by the composition of the folding operations, and notice that by Proposition \ref{folding2} this is a pointed homotopy equivalence: a pointed homotopy inverse can be built following the chain of folding operations, and is given by $q:(R',*)\rar(G,*)$ which sends the edge $e_1$ to the path $\mu_g$, the edge $e_2$ to the path $\mu_{h_1}\ol\mu_g$ and the edge $e_3$ to the path $\ol\mu_g\mu_{h_2}\mu_g\mu_g\ol\mu_{h_1}$. In order to obtain generators for the kernel $\fI_g\sgr H*\gen{x}$ we have to look at the image $q_*(e_3\ol{e}_2)$: we obtain that the kernel is generated (as a normal subgroup) by just one equation $\fI_g=\ggen{\ol{x}h_2xx\ol{h}_1x\ol{h}_1}$.

\begin{figure}[h!]
\centering
\begin{tikzpicture}[scale=1]
\node (1) at (0,0) {$*$};
\node (2) at (-2,1) {.};
\node (3) at (-2,-1) {.};
\node (4) at (-4,-1) {.};

\draw[->,out=170,in=-50,looseness=1] (1) to node[below]{$b$} (2);
\draw[->,out=0,in=120] (2) to node[above]{$a$} (1);
\draw[->] (1) to node[below]{$a$} (3);
\draw[->,out=200,in=-20,looseness=1] (3) to node[below]{$b$} (4);
\draw[->,out=20,in=160,looseness=1] (4) to node[above]{$b$} (3);

\draw[->,out=40,in=-40,looseness=13] (1) to node[right]{$a$} (1);
\end{tikzpicture}
\caption{In the picture we can see the graph $G$ of Example \ref{example46}. Here $H=\gen{ba,ab^2\ol{a}}$ and $g=a$. On the left of the basepoint we have the graph $\bcore{H}$, while on the right we have $\bcore{\gen{g}}$.}
\label{example46G}
\end{figure}
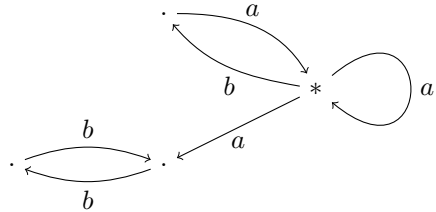

\begin{figure}[h!]
\centering
\begin{tikzpicture}[scale=1]
\node (1) at (0,0) {$*$};
\node (2) at (-2,1) {.};
\node (3) at (-2,-1) {.};
\node (4) at (-4,-1) {.};

\draw[->,out=40,in=-40,looseness=13,line width=0.5mm] (1) to node[right]{$a$} (1);
\draw[->,out=170,in=-50,looseness=1] (1) to node[below]{$b$} (2);
\draw[->,out=0,in=120,line width=0.5mm] (2) to node[above]{$a$} (1);
\draw[->] (1) to node[below]{$a$} (3);
\draw[->,out=200,in=-20,looseness=1] (3) to node[below]{$b$} (4);
\draw[->,out=20,in=160,looseness=1] (4) to node[above]{$b$} (3);

\node (1p) at (0,-3) {$*$};
\node (3p) at (-2,-4) {.};
\node (4p) at (-4,-4) {.};

\draw[->,out=40,in=-40,looseness=13,line width=0.5mm] (1p) to node[right]{$a$} (1p);
\draw[->,out=160,in=80,looseness=18] (1p) to node[left]{$b$} (1p);
\draw[->,line width=0.5mm] (1p) to node[below]{$a$} (3p);
\draw[->,out=200,in=-20,looseness=1] (3p) to node[below]{$b$} (4p);
\draw[->,out=20,in=160,looseness=1] (4p) to node[above]{$b$} (3p);

\node (1q) at (0,-6) {$*$};
\node (4q) at (-2,-7) {.};

\draw[->,out=40,in=-40,looseness=13] (1q) to node[right]{$a$} (1q);
\draw[->,out=160,in=80,looseness=18,line width=0.5mm] (1q) to node[left]{$b$} (1q);
\draw[->,out=230,in=10,looseness=1] (1q) to node[below]{$b$} (4q);
\draw[->,out=50,in=190,looseness=1,line width=0.5mm] (4q) to node[above]{$b$} (1q);

\node (1r) at (0,-9) {$*$};

\draw[->,out=40,in=-40,looseness=13] (1r) to node[right]{$a$} (1r);
\draw[->,out=160,in=80,looseness=18] (1r) to node[left]{$b$} (1r);
\draw[->,out=280,in=200,looseness=18] (1r) to node[left]{$b$} (1r);

\node (a) at (3,0) {$G$};
\node (b) at (3,-9) {$R'$};
\end{tikzpicture}
\caption{A maximal rank-preserving folding sequence for the graph $G$ of Example \ref{example46}. At each step, we highlight in bold the two edges that are going to be folded in the next step.}\label{example46fold}
\end{figure}
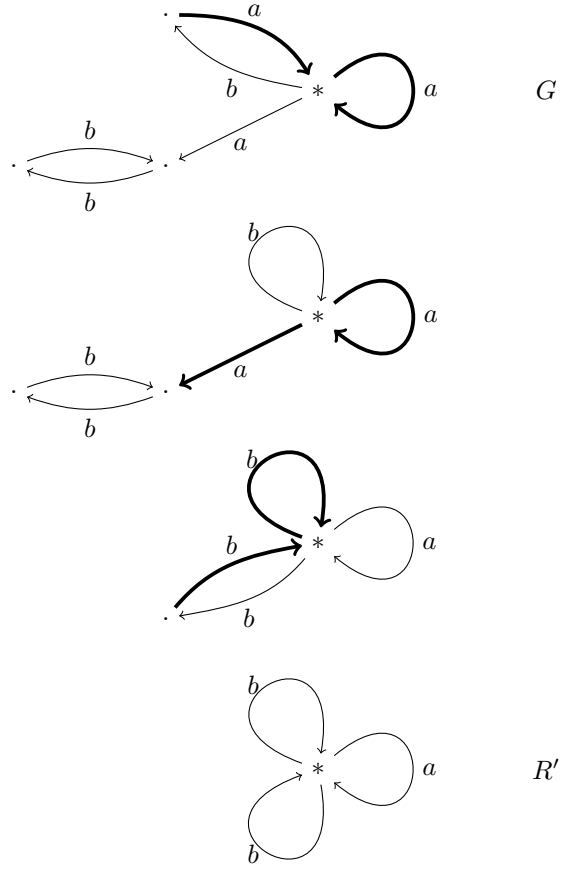

We observe that this unique generator has even degree, and thus Lemma \ref{evendegree} tells us that every equation in $\fI_g$ has even degree. We shall explain why there is no equation of degree $2$, there is exactly one equation of degree $4$ up to conjugation and inverse and there are equations of degree $6$. By Lemma \ref{obtainingdegrees} it follows that $d_{min}=4$ and $D_g=2\bN\setminus\{0,2\}$.

\begin{myrmk}
In order characterize all the equations of degrees $2,4,6$, we could use Theorem \ref{main4}; this is too long to do by hands, but quite easy to do with the aid of a computer. It is also possible to prove them with some combinatorics of the cancellation between words; we do not provide a full proof here, but rather a sketch.

Consider the map between free groups $\psi:\gen{h_1,h_2,x}\rar\gen{h_1,x}$ with $\psi(h_1)=h_1,\psi(x)=x,\psi(h_2)=xh_1\ol{x}h_1\ol{x}\ol{x}$ and notice that $\fI_g=\ker\psi$. Up to conjugation, an equation $w\in\fI_g$ can be written as reduced word $w(h_1,h_2,x)=u_1(h_1,h_2)x^{e_1}...u_r(h_1,h_2)x^{e_r}$. We now substitute each occurrence of $h_2$ with $xh_1\ol{x}h_1\ol{x}\ol{x}$, and each occurrence of $\ol{h_2}$ with $xx\ol{h}_1x\ol{h}_1\ol{x}$, and after this substitution we reduce the obtained word, until we get the trivial word. During the reduction process, each block $xh_1\ol{x}h_1\ol{x}\ol{x}$ and $xx\ol{h}_1x\ol{h}_1\ol{x}$, obtained from an occurrence of $h_2$ or $\ol{h}_2$, will completely cancel at some point: we take the occurrence of $h_2$ such that the corresponding block is the first to completely cancel during the reduction process. We now look at the word $w$ near that occurrence of $h_2$ or $\ol{h}_2$, and we obtain that $w$ contains at least one of
$$h_2xx\ol{h}_1x,\
\ol{x}h_2xx,\
x\ol{h}_1\ol{x}h_2x,\
x\ol{h}_1x\ol{h}_1\ol{x}h_2,\
h_2xx\ol{h}_1\ol{x}\ol{h}_2,\
\ol{h}_2xx\ol{h}_1\ol{x}h_2$$
(or of their inverses) as a subword. These can be substituted with (respectively)
$$xh_1,\ h_1\ol{x}h_1,\ h_1\ol{x},\ \ol{x},\ xx\ol{h}_1\ol{x},\ xx\ol{h}_1\ol{x}$$
in order to get a shorter (possibly not reduced) equation.

This immediately implies that $\fI_g$ contains no equation of degree $2$, and it also allows to deduce that the only equations of degree $4$ are the conjugates of the generator. With some more work, it is also possible to give a characterization of all the degree $6$ equations.
\end{myrmk}

%
%\begin{mylemma}
%The subgroup $\fI_g=\ggen{\ol{x}h_2xx\ol{h}_1x\ol{h}_1}$ contains no equation of degree $2$.
%\end{mylemma}
%\begin{proof}
%Consider the free group $\gen{h_1,h_2,x}$ with basis $h_1,h_2,x$ and the free group $\gen{h_1,x}$ with basis $h_1,x$. Let $\psi:\gen{h_1,h_2,x}\rar\gen{h_1,x}$ be given by $\psi(h_1)=h_1$ and $\psi(x)=x$ and $\psi(h_2)=xh_1\ol{x}h_1\ol{x}\ol{x}$; we observe that $\ker\psi=\fI_g$. Suppose by contradiction we have an equation $w\in\fI_g$ of degree $2$: up to conjugation, and possibly up to taking the inverse, we can assume that $w$ is of the form $w=uxvx^{\pm1}$ for some reduced words $u=u(h_1,h_2)$ and $v=v(h_1,h_2)$.
%
%The condition $w\in\fI_g=\ker\psi$ means that $u(h_1,xh_1\ol{x}h_1\ol{x}\ol{x})\cdot x\cdot v(h_1,xh_1\ol{x}h_1\ol{x}\ol{x})\cdot x=1$ as element of the free group $\gen{h_1,x}$. Take a reduced writing for $u(h_1,h_2)xv(h_1,h_2)x^{\pm1}$, and substitute each occurrence of $h_2$ with $xh_1\ol{x}\bb{h_1}\ol{x}\ol{x}$, and each occurrence of $\ol{h_2}$ with $xx\bb{\ol{h}_1}x\ol{h}_1\ol{x}$, writing in bold the some of the $h_1$ letters, as shown, and in normal format the other letters. We now perform a reduction process on the word that we obtained, and the key observation is the following: no matter how we decide to reduce the word, no bold letter ever gets canceled. In particular, since $u(h_1,h_2)xv(h_1,h_2)x^{\pm1}$ contains at least one occurrence of $h_2$, we can't obtain the trivial word at the end of the reduction process, contradiction.
%\end{proof}

\subsection{An ideal with both even-degree and odd-degree equations}\label{example23}

Consider the subgroup $H=\gen{h_1,h_2}$ with $h_1=b$ and $h_2=ababa$ and the element $g=a$. We see the corresponding graph $G$ in figure \ref{example23G}. We can now proceed as in Example \ref{example46}: we choose a maximal sequence of rank-preserving folding operations for $G$, we build an homotopy inverse to the sequence of folding operations, we obtain a generator for the normal subgroup $\fI_g\nor H*\gen{x}$. Whatever sequence of folding operations you choose, you will always get the same generator, up to inverse and cyclic permutations, namely $\fI_g=\ggen{\ol{h}_2xh_1xh_1x}$.

\begin{figure}[h!]
\centering
\begin{tikzpicture}[scale=1]
\node (1) at (0,0) {$*$};
\node (2) at (-1.5,0) {.};
\node (3) at (-3,0) {.};
\node (4) at (-3,-1.5) {.};
\node (5) at (-1.5,-1.5) {.};

\draw[->,out=160,in=90,looseness=20] (1) to node[above]{$b$} (1);
\draw[->] (1) to node[below]{$a$} (2);
\draw[->] (2) to node[below]{$b$} (3);
\draw[->] (3) to node[left]{$a$} (4);
\draw[->] (4) to node[above]{$b$} (5);
\draw[->] (5) to node[below]{$a$} (1);

\draw[->,out=40,in=-40,looseness=13] (1) to node[right]{$a$} (1);
\end{tikzpicture}
\caption{In the picture we can see the graph $G$ of Example \ref{example23}. Here $H=\gen{b,ababa}$ and $g=a$. On the left of the basepoint we have the graph $\bcore{H}$. On the right of the basepoint we have the graph $\bcore{\gen{g}}$.}
\label{example23G}
\end{figure}

We have that $\fI_g$ contains no equation of degree $1$. It contains equations of degree $2$, which are exactly the ones of the form $[(h_2h_1)^i,xh_1]$ for $i\not=0$ up to conjugation and inverses. It also contains equations of degree $3$ (possibly essentially different from the generator). By Lemma \ref{obtainingdegrees} it follows that $d_{min}=2$ and $D_g=\{d : d\ge2\}$. We observe that the equations of minimum possible degree are not enough in this case to generate the whole ideal: in fact, according to Lemma \ref{evendegree}, equations of degree $2$ generate a normal subgroup containing only even-degree equations.

\begin{myrmk}
As in Example \ref{example46}, it is possible to characterize equations of degree $2$ and $3$ with Theorem \ref{main4}, using a computer, or it is possible to do it by hands with some combinatorics of the cancellations inside the words; and again, for this second method, we provide a sketch below.

Consider the map between free groups $\psi:\gen{h_1,h_2,x}\rar\gen{h_1,x}$ with $\psi(h_1)=h_1,\psi(x)=x,\psi(h_2)=xh_1xh_1x$ and notice that $\fI_g=\ker\psi$. Up to conjugation, an equation $w\in\fI_g$ can be written as reduced word $w(h_1,h_2,x)=u_1(h_1,h_2)x^{e_1}...u_r(h_1,h_2)x^{e_r}$. We now substitute each occurrence of $h_2$ with $xh_1xh_1x$ and each occurrence of $\ol{h_2}$ with $xh_1xh_1x$; as in Example \ref{example46} we take the occurrence of $h_2$ or $\ol{h}_2$ such that the corresponding $xh_1xh_1x$ or $\ol{x}\ol{h}_1\ol{x}\ol{h}_1\ol{x}$ is the first to completely cancel, and we look at the word $w$ near that occurrence of $h_2$ or $\ol{h}_2$. We obtain that $w$ contains at least one of
$$\ol{h}_2xh_1x,\ x\ol{h}_2x,\ xh_1x\ol{h}_2,\ \ol{h}_2xh_1h_2,\ h_2h_1x\ol{h}_2,$$
(or of their inverses) as a subword. These can be substituted with (respectively)
$$\ol{x}\ol{h}_1,\ \ol{h}_1\ol{x}\ol{h}_1,\ \ol{h}_1\ol{x},\ h_1x,\ xh_1$$
in order to get a shorter (possibly not reduced) equation.

Dealing with some cases it can be proved that equations of degree $2$ are exactly the ones of the form $[(h_2h_1)^i,xh_1]$ for $i\not=0$ up to conjugation and inverses, and one can produce equations of degree $3$ which are essentially different from the generator.
\end{myrmk}

\subsection{An ideal with two generators}\label{example234}

Consider the subgroup $H=\gen{h_1,h_2,h_3}$ with $h_1=a^2\ol{b}\ol{a}$ and $h_2=a^3$ and $h_3=ba\ol{b}$ and the element $g=a^2\ol{b}$. We can see the corresponding graph $G$ in figure \ref{example234G} and a maximal sequence of rank-preserving folding operations in figure \ref{example234fold}. The group $\pi_1(G,*)$ is a free group with four generators $[\mu_{h_1}],[\mu_{h_2}],[\mu_{h_3}],[\mu_g]$, which are the homotopy classes of the reduced paths $\mu_{h_1},\mu_{h_2},\mu_{h_3},\mu_g$ corresponding to the elements $h_1,h_2,h_3\in H$ and $g$ respectively. At the end of the sequence of folding operations we obtain a rose $R'$ with one $b$-labeled edge $e_1$ and three $a$-labeled edges $e_2,e_3,e_4$. The map $p:(G,*)\rar(R',*)$ given by the composition of the folding operations is an homotopy equivalence, according to Proposition \ref{folding2}, and a pointed homotopy inverse is $q:(R',*)\rar(G,*)$ which sends the edge $e_1$ to the path $\mu_{h_3}\ol\mu_g\ol\mu_{h_1}\mu_g$, the edge $e_2$ to the path $\ol\mu_g\mu_{h_1}\mu_g\ol\mu_{h_3}\ol\mu_g\mu_{h_2}$, the edge $e_3$ to the path $\ol\mu_{h_1}\mu_g$ and the edge $e_4$ to the path $\ol\mu_g\mu_{h_1}\mu_g\mu_{h_3}\ol\mu_g\ol\mu_{h_1}\mu_g$. We look at the images $q_*(e_2\ol{e}_3)$ and $q_*(e_4\ol{e}_3)$ and we obtain that the kernel $\fI_g\sgr H*\gen{x}$ is generated (as normal subgroup) by the equations $\fI_g=\ggen{\ol{x}h_1x\ol{h}_3\ol{x}h_2\ol{x}h_1,\ol{x}h_1xh_3\ol{x}}$.

\begin{figure}[h!]
\centering
\begin{tikzpicture}[scale=1]
\node (1) at (0,0) {*};
\node (2) at (-2,1.5) {.};
\node (3) at (-2,0) {.};
\node (4) at (-2,-1) {.};
\node (5) at (2,1) {.};
\node (6) at (2,-1) {.};

\draw[->] (1) to node[above]{$a$} (2);
\draw[->] (2) to node[right]{$a$} (3);
\draw[->,out=-120,in=120,looseness=1] (2) to node[left]{$b$} (3);
\draw[->] (3) to node[above]{$a$} (1);
\draw[->] (1) to node[below]{$b$} (4);
\draw[->,out=120,in=180,looseness=25] (4) to node[left]{$a$} (4);

\draw[->] (1) to node[below]{$a$} (5);
\draw[->] (5) to node[right]{$a$} (6);
\draw[->] (1) to node[below]{$b$} (6);
\end{tikzpicture}
\caption{In the picture we can see the graph $G$ of Example \ref{example234}. Here $H=\gen{a^2\ol{b}\ol{a},a^3,ba\ol{b}}$ and $g=a^2\ol{b}$. On the left of the basepoint we have the graph $\bcore{H}$. On the right of the basepoint we have the graph $\bcore{\gen{g}}$.}
\label{example234G}
\end{figure}

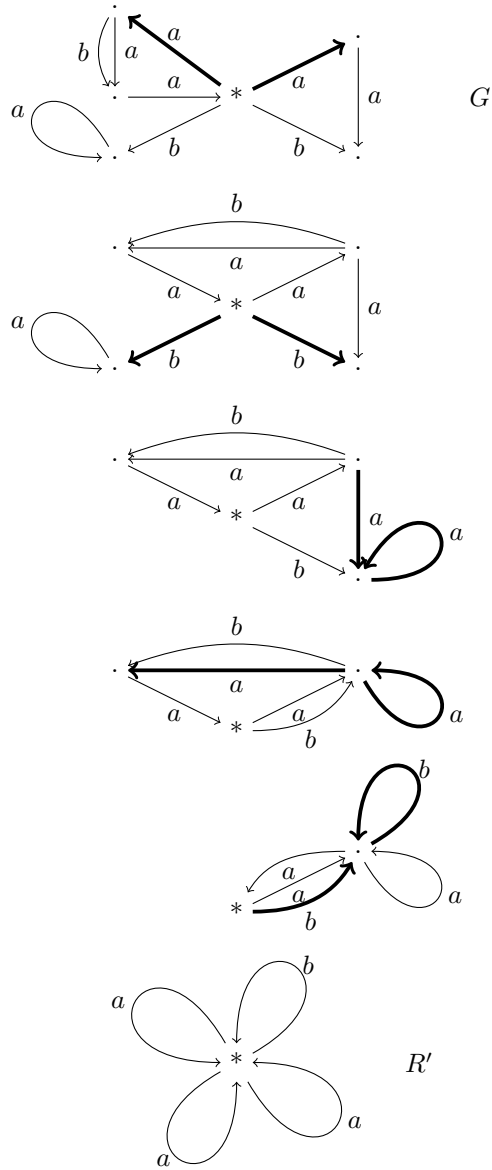
\begin{figure}[h!]
\centering
\begin{tikzpicture}[scale=0.8]
\node (1) at (0,0) {*};
\node (2) at (-2,1.5) {.};
\node (3) at (-2,0) {.};
\node (4) at (-2,-1) {.};
\node (5) at (2,1) {.};
\node (6) at (2,-1) {.};

\draw[->,line width=0.5mm] (1) to node[above]{$a$} (2);
\draw[->] (2) to node[right]{$a$} (3);
\draw[->,out=-120,in=120,looseness=1] (2) to node[left]{$b$} (3);
\draw[->] (3) to node[above]{$a$} (1);
\draw[->] (1) to node[below]{$b$} (4);
\draw[->,out=120,in=180,looseness=25] (4) to node[left]{$a$} (4);
\draw[->,line width=0.5mm] (1) to node[below]{$a$} (5);
\draw[->] (5) to node[right]{$a$} (6);
\draw[->] (1) to node[below]{$b$} (6);

\node (1) at (0,-3.5) {*};
\node (3) at (-2,-2.5) {.};
\node (4) at (-2,-4.5) {.};
\node (5) at (2,-2.5) {.};
\node (6) at (2,-4.5) {.};

\draw[->] (5) to node[below]{$a$} (3);
\draw[->,out=160,in=20,looseness=1] (5) to node[above]{$b$} (3);
\draw[->] (3) to node[below]{$a$} (1);
\draw[->] (1) to node[below]{$a$} (5);
\draw[->] (5) to node[right]{$a$} (6);
\draw[->,line width=0.5mm] (1) to node[below]{$b$} (6);
\draw[->,line width=0.5mm] (1) to node[below]{$b$} (4);
\draw[->,out=120,in=180,looseness=25] (4) to node[left]{$a$} (4);

\node (1) at (0,-7) {*};
\node (3) at (-2,-6) {.};
\node (5) at (2,-6) {.};
\node (6) at (2,-8) {.};

\draw[->] (5) to node[below]{$a$} (3);
\draw[->,out=160,in=20,looseness=1] (5) to node[above]{$b$} (3);
\draw[->] (3) to node[below]{$a$} (1);
\draw[->] (1) to node[below]{$a$} (5);
\draw[->,line width=0.5mm] (5) to node[right]{$a$} (6);
\draw[->] (1) to node[below]{$b$} (6);
\draw[->,out=0,in=60,looseness=25,line width=0.5mm] (6) to node[right]{$a$} (6);

\node (1) at (0,-10.5) {*};
\node (3) at (-2,-9.5) {.};
\node (5) at (2,-9.5) {.};

\draw[->,line width=0.5mm] (5) to node[below]{$a$} (3);
\draw[->,out=160,in=20,looseness=1] (5) to node[above]{$b$} (3);
\draw[->] (3) to node[below]{$a$} (1);
\draw[->] (1) to node[below]{$a$} (5);
\draw[->,out=0,in=-120,looseness=1] (1) to node[below]{$b$} (5);
\draw[->,out=-60,in=0,looseness=25,line width=0.5mm] (5) to node[right]{$a$} (5);

\node (1) at (0,-13.5) {*};
\node (5) at (2,-12.5) {.};

\draw[->,out=30,in=90,looseness=25,line width=0.5mm] (5) to node[right]{$b$} (5);
\draw[->,out=180,in=60,looseness=1] (5) to node[below]{$a$} (1);
\draw[->] (1) to node[below]{$a$} (5);
\draw[->,out=0,in=-120,looseness=1,line width=0.5mm] (1) to node[below]{$b$} (5);
\draw[->,out=-60,in=0,looseness=25] (5) to node[right]{$a$} (5);

\node (1) at (0,-16) {*};

\draw[->,out=30,in=90,looseness=20] (1) to node[right]{$b$} (1);
\draw[->,out=120,in=180,looseness=20] (1) to node[left]{$a$} (1);
\draw[->,out=210,in=-90,looseness=20] (1) to node[left]{$a$} (1);
\draw[->,out=-60,in=0,looseness=20] (1) to node[right]{$a$} (1);

\node (a) at (4,0) {$G$};
\node (b) at (3,-16) {$R'$};
\end{tikzpicture}
\caption{A maximal rank-preserving folding sequence for the graph $G$ of Example \ref{example234}. At each step, we highlight in bold the two edges that are going to be folded in the next step.}
\label{example234fold}
\end{figure}

It is easy to show that $\fI_g$ contains equations of degree $2$, but not equations of degree $1$. It follows that $d_{min}=2$ and $D_g=\bN\setminus\{0,1\}$. We observe that, despite $d_{min}=2$, equations of degree $2$ are not enough to generate the whole ideal $\fI_g$; in fact, the equations of degree $2$ generate a normal subgroup containing only even-degree equations (see Lemma \ref{evendegree}), while $\fI_g$ also contains equations of odd degree.

\section{Equations in more variables}\label{SectionMultivariate}

We point out that most of the results of this paper can be generalized to equations in more than one variable. Let $F_n$ be a free group generated by $n$ elements $a_1,...,a_n$. Let $H\sgr F_n$ be a finitely generated subgroup and let $\gen{x_1},\gen{x_2},...,\gen{x_m}$ be infinite cyclic groups.

\begin{mydef}
An \textbf{equation} with coefficients in $H$ is an element $w\in H*\gen{x_1}*...*\gen{x_m}$.
\end{mydef}

\begin{mydef}
Define the \textbf{multi-degree} of an equation $w\in H*\gen{x_1}*...*\gen{x_m}$ as the $m$-tuple $(d_1,...,d_m)$ of integer numbers, where $d_i$ is the number of occurrences of $x_i$ and of $\ol{x}_i$ in the cyclic reduction of $w$.
\end{mydef}

For $(g_1,...,g_m)\in(F_n)^m$ we define the map $\varphi_{g_1,...,g_m}:H*\gen{x_1}*...*\gen{x_m}\rar F_n$ such that $\restr{\varphi_{g_1,...,g_m}}{H}$ is the inclusion and $\varphi_{g_1,...,g_m}(x_i)=g_i$ for $i=1,...,m$.

\begin{mydef}
We say that an $m$-tuple $(g_1,...,g_m)\in(F_n)^m$ is a \textbf{solution} to the equation $w\in H*\gen{x_1}*...*\gen{x_m}$ if $w\in\ker\varphi_{g_1,...,g_m}$.
\end{mydef}

\begin{mydef}
For $(g_1,...,g_m)\in(F_n)^m$ we define the \textbf{ideal} $\fI_{g_1,...,g_m}$ to be the normal subgroup $\fI_{g_1,...,g_m}=\ker\varphi_{g_1,...,g_m}\nor H*\gen{x_1}*...*\gen{x_m}$.
\end{mydef}

\begin{mydef}
We say that $(g_1,...,g_m)\in(F_n)^m$ \textbf{depends on $H$} if $\fI_{g_1,...,g_m}$ is non-trivial.
\end{mydef}

Fix now an $m$-tuple $(g_1,...,g_m)\in(F_n)^m$. As we did in the one-variable case, we now want to see equations as paths in a suitable graph. Let $G=\bcore{H}\vee\bcore{\gen{g_1}}\vee...\vee\bcore{\gen{g_m}}$ be the $\grafo$ given by the disjoint union of $\bcore{H},\bcore{\gen{g_1}},...,\bcore{\gen{g_m}}$, where we identify all the basepoints to a unique point. Let also $f:G\rar R_n$ be the labeling map.

With the same argument as in Theorem \ref{idealfingen} we can prove the following theorem:

\begin{mythm}
We have the following:

(i) The ideal $\fI_{g_1,...,g_m}\nor H*\gen{x_1}*...*\gen{x_m}$ is finitely generated as a normal subgroup.

(ii) The set of generators for $\fI_{g_1,...,g_m}$ can be taken to be a subset of a basis for $H*\gen{x_1}*...*\gen{x_m}$.

(iii) There is an algorithm that, given $H$ and $g_1,...,g_m$, computes a finite set of normal generators for $\fI_{g_1,...,g_m}$ which is also a subset of a basis for $H*\gen{x_1}*...*\gen{x_m}$.
\end{mythm}

We have an isomorphism $\theta:H*\gen{x_1}*...*\gen{x_m}\rar\pi_1(G,*)$ so that we can define the same correspondence as in Definitions \ref{defcorrpath} and \ref{defcorrequation}. Non-trivial equations $w\in\fI_{g_1,...,g_m}$ correspond to reduced paths $\sigma:I_l\rar G$ with $\sigma(0)=\sigma(1)=*$ and such that $f\circ\sigma$ is homotopically trivial (relative to its endpoints). The following three lemmas relate the degree of an equation to its corresponding path, and the proofs are exactly the same as for Lemmas \ref{degree}, \ref{innermostcouple} and \ref{innermostbounded}.

\begin{mylemma}
Let $\sigma:I_l\rar G$ be a cyclically reduced path. For $i=1,...,m$ let $e_i$ be any edge of $G$ that belongs to the subgraph $\core{\gen{g_i}}$. Let also $(d_1,...,d_m)$ be the multi-degree of the equation $w$ corresponding to $\sigma$. Then the path $\sigma$ crosses the edge $e_i$ exactly $d_i$ times (in either direction).
\end{mylemma}

\begin{mylemma}
Let $\sigma:I_l\rar G$ be a reduced path with $\sigma(0)=\sigma(1)=*$ and such that $f\circ\sigma$ is homotopically trivial (relative to its endpoints); let $(s_1,t_1),...,(s_{l/2},t_{l/2})$ be a maximal reduction process for $f\circ\sigma$. Let $(s_i,t_i)$ be an innermost cancellation. Then, among $\sigma(s_i)$ and $\sigma(t_i)$, one is an edge of $\bcore{H}$ and the other is an edge of $\bcore{\gen{g_1}}\vee...\vee\bcore{\gen{g_m}}$.
\end{mylemma}

\begin{mylemma}
Let $\sigma:I_l\rar G$ be a cyclically reduced path with $\sigma(0)=\sigma(1)=*$ and such that $f\circ\sigma$ is homotopically trivial (relative to its endpoints); let $(s_1,t_1),...,(s_{l/2},t_{l/2})$ be a maximal cancellation process for $f\circ\sigma$. Let $(d_1,...,d_m)$ be the multi-degree of the equation $w\in\fI_{g_1,...,g_m}$ corresponding to $\sigma$. Then the reduction process contains at most $2(d_1+...+d_m)$ innermost cancellations.
\end{mylemma}

We can define a parallel cancellation move, in the exact same way as in Definition \ref{parallel} and Lemma \ref{parallelcancellation}. Lemma \ref{degreepreserving} about degree-preserving parallel cancellation moves remains true too, where degree-preserving means that the equations $w,w'$, before and after the parallel cancellation move respectively, have the same multi-degree $(d_1,...,d_m)=(d_1',...,d_m')$.

It is also possible to define a parallel insertion move, exactly as in Definition \ref{insertionpath} and Lemma \ref{parallelinsertion}. Lemmas \ref{parallelinverses}, \ref{parallelinverses2}, \ref{insertioncommutes}, \ref{insertionsame}, \ref{insertioninsertion} remain true.

In the exact same way as we proved Proposition \ref{findingparallels2} and Theorems \ref{main3} and \ref{main4}, we can prove the following.

\begin{myprop}
Let $w\in\fI_{g_1,...,g_m}$ be a cyclically reduced equation of multi-degree $(d_1,...,d_m)$ and let $\sigma:I_l\rar G$ be the corresponding path; let $(s_1,t_1),...,(s_{l/2},t_{l/2})$ be any maximal reduction process for $f\circ\sigma$. Suppose $l>32L^4(d_1+...+d_m)^2+16L^3(d_1+...+d_m)$. Then the reduction process contains two $\parallele$ couples $(s_\alpha,t_\alpha),(s_\beta,t_\beta)$ such that the corresponding parallel cancellation move is degree-preserving.
\end{myprop}

\begin{mythm}
Suppose $\fI_{g_1,...,g_m}$ contains a non-trivial equation of degree $(d_1,...,d_m)$. Then $\fI_{g_1,...,g_m}$ contains non-trivial equation $w$ of degree $(d_1,...,d_m)$ such that the corresponding path $\sigma:I_l\rar G$ has length $l\le 32L^4(d_1+...+d_m)^2+16L^3(d_1+...+d_m)$.
\end{mythm}

\begin{mycor}
There is an algorithm that, given $H,g_1,...,g_m$ and an $m$-tuple $(d_1,...,d_m)$ of non-negative integers, tells us whether $\fI_{g_1,...,g_m}$ contains non-trivial equations of multi-degree $(d_1,...,d_m)$, and, if so, produces an equation $w\in\fI_{g_1,...,g_m}$ of multi-degree $(d_1,...,d_m)$.
\end{mycor}

\begin{mythm}
Let $w\in\fI_{g_1,...,g_m}$ be a cyclically reduced equation of multi-degree $(d_1,...,d_m)$ and let $\sigma:I_l\rar G$ be the corresponding path. Then there is a cyclically reduced equation $w'\in\fI_{g_1,...,g_m}$ of degree $(d_1,...,d_m)$ with corresponding path $\sigma':I_{l'}\rar G$, and a maximal reduction process for $\sigma'$, such that:

(i) The path $\sigma'$ has length $l'\le 32L^4(d_1+...+d_m)^2+16L^3(d_1+...+d_m)$.

(ii) The path $\sigma$ can be obtained from $\sigma'$ by means of at most $l'/2$ insertion moves, each of them performed on a distinct couple of edges of $I_{l'}$.
\end{mythm}

%Finally, we can consider the set
%$$D_{g_1,...,g_m}=\{(d_1,...,d_m)\in\bN^m :\text{ there is a non-trivial equation }w\in\fI_{g_1,...,g_m}\text{ of multi-degree }(d_1,...,d_m)\}$$
%and partition $D_{g_1,...,g_m}=D^\ex_{g_1,...,g_m}\sqcup\bigsqcup_{k\in\bN}D^{\pol,k}_{g_1,...,g_k}$ according to the growth rate of the function $\rho_{g_1,...,g_m,(d_1,...,d_m)}$.

\bibliographystyle{alpha}
\nocite{*}
\bibliography{bibliography.bib}

\end{document}